\documentclass[12pt]{article}

\textheight=8.5in \topmargin=-0.3in \everymath{\displaystyle}
\usepackage{graphics}
\usepackage{amsmath}
\usepackage{amsfonts}
\usepackage{amssymb}
\usepackage{graphicx}
\input xy
\xyoption{all}
\usepackage{color}
\usepackage{graphicx,amssymb,amsfonts,latexsym,amsmath,amsthm}

\usepackage{amsfonts}
\usepackage{amssymb}
\usepackage{eucal}

\newtheorem{theorem}{Theorem}[section]
\newtheorem{lemma}[theorem]{Lemma}
\newtheorem{proposition}[theorem]{Proposition}
\newtheorem{definition}[theorem]{Definition}

\newtheorem{corollary}[theorem]{Corollary}
\newtheorem{remark}[theorem]{Remark}

\def\cR{{\mathcal R}}

\def\C{{\mathbb{C}}}

\everymath{\displaystyle}
\newcommand{\RE} {{\rm I \kern-2.8pt R} }

\newcommand{\beasnum}{\begin{eqnarray}}
\newcommand{\eeasnum}{\end{eqnarray}}
\newcommand{\beas}{\begin{eqnarray*}}
\newcommand{\eeas}{\end{eqnarray*}}
\begin{document}

\begin{center}
{\Large {\bf Evolutionary Game Theory on Measure Spaces: Asymptotic Behavior of Solutions}} {\large
\textbf{\bigskip}}

\vspace{0.1in}

John Cleveland$^{\dag}$ and Azmy S. Ackleh$^{\ddag}$

\vspace{0.1in}

$^\dag$Department of Mathematics\\
Penn State University\\
University Park, State College, PA 16802\\

\vspace{0.1in}

$^\ddag$Department of Mathematics\\
University of Louisiana at Lafayette\\
Lafayette, Louisiana 70504-1010\\

\vspace{0.1in}
\end{center}

\begin{abstract}
  {In \cite{CLEVACK}  we formulated an evolutionary game theory model as a dynamical system on the state space of finite signed Borel measures under the weak* topology. We showed that this  model is well-posed, continuous and rich enough to treat selection and mutation, discrete and continuous initial conditions and to encompass all classical nonlinearities. The focus of this paper is  to extend the analysis to include the long-time behavior of solutions to the model. In particular, we we show that $M(Q)$, the finite signed Borel measures are asymptotically closed. This means that if the initial condition is a finite signed Borel measure and if the asymptotic limit of the model solution exists, then it will be a measure (note that function spaces such as $L^1(Q)$ and $C(Q)$ do not have this property). We also establish permanence results for the full replicator-mutator model. Then, we study the asymptotic analysis in the case where there is more than one strategy of a given fitness (a continuum of strategies of a given fitness), a case that often arises in applications.  To study this case our mathematical structure must include the ability to demonstrate the convergence of the model solution to a measure supported on a continuum of strategies.
For this purpose, we demonstrate how to perform
completions of the space of measures and how to use these completions
to formulate weak (generalized) asymptotic limits.  In particular, we show that for the pure replicator dynamics the (weak) solution of the dynamical system converges to a Dirac measure centered at the fittest strategy class; thus this Dirac measure is a globally attractive equilibrium point which is termed a continuously stable strategy
(CSS). It is also shown that in the discrete case of the pure
replicator dynamics and even for small perturbation of the pure
replicator dynamics (i.e., selection with small mutation) there exists a globally asymptotically stable
equilibrium.}\\


\end{abstract}

\section{ Introduction}
Game theory (GT) is an analytical tool used to model strategic
interactions or conflict situations. Any situation that requires one
to anticipate a rival's response to one's action is a potential
context for GT. Game theory comprises players, the strategies of
those players and a payoff for each player which is dependent upon
all strategies played. In \cite{Maynard2} Maynard Smith and Price
saw evolutionary biology as game theory and provided a critical link
to defining an evolutionary game. Essentially in evolutionary
biology there exists populations of \textbf{reproducing} organisms.
In the reproduction process some reproduce themselves exactly or due
to the biotic and abiotic environment some organisms increase in
more abundance than others and this implies that certain individuals
are being \textbf{selected} over other individuals. At other times
we have errors occur in this selection process and these are called
\textbf{mutations}. According to \cite{BrownMcGill} in the
evolutionary game, individual organisms are the players, their
heritable phenotypes or behaviors are their strategies, and their
(per capita) growth rates (fitness) are their payoffs. The insight
of John Maynard Smith was to equate fitness with payoff. So the
payoff to a particular strategy is defined as the expected numerical
contribution to the next generation (the growth rate or time rate of
change of the size of the strategy class).

So evolutionary game theory (EGT) is the creation and study of  mathematical models that describe how the strategy profiles
in games change over time due to mutation and selection
(replication).
Many researchers have recently devoted their attention to
the study of such EGT models (e.g.
\cite{AFT,AMFH,calsina,CALCAD,Hof,Hof2, MagWeb,SES}). To date almost
all EGT models are formulated as {\it density} models
\cite{AMFH,calsina,CALCAD, MagWeb,SES} with {\it linear} mutation
term. There are several formulations of pure selection or replicator
equation dynamics on measure spaces \cite{AFT,Bomze,CressHof}. The
recent formulations of selection-mutation balance equations on the
probability measures by \cite{EMP,KKO} are novel constructions.
These models describe the aging of an infinite population as a
process of accumulation of mutations in a genotype. The dynamical
equation which describes the system is of Kimura-Maruyama type. Thus
far in selection-mutation studies the mutation process has been
modeled using two different approaches: (1) A diffusion type
operator \cite{GVA,SES}; (2) An integral type operator that makes use of
a mutation kernel \cite{AFT,calsina,CALCAD,EMP, KKO}. Here we focus
on the second approach for modeling mutation.

It is well known that the solutions of many such models constructed on the state space of continuous or integrable functions converge to a Dirac measure centered at the fittest class \cite{AFT,AMFH,calsina,CALCAD,GVA,GR1,P,GR2}. This convergence is in the weak$^* $ topology \cite{AMFH}. Thus, the asymptotic limit of the solution is not in such
 state spaces; it is a measure.  How these measures arise naturally in a biological and adaptive dynamics environment is illustrated quite well in \cite[chpt.2]{P}. Thus, in \cite{AFT} the authors formulated a {\it pure} selection model with density dependent birth and mortality
function and a 2-dimensional trait space on the space of finite signed measures. They discussed
 existence-uniqueness of solutions and studied the long term behavior of the model.
Here, we substantially generalize the results in that paper in several directions: 1) we model selection and mutation
and we allow for the mutation kernel to be a family of measures (thus simultaneously treating discrete and continuous
strategy spaces) 2) we consider more general nonlinearities and thus the results apply to a wider class of models 3) we
allow for more than one fittest strategy.

To motivate our attempt of allowing more than one fittest strategy recall that in  \cite{AMFH} the authors considered the following logistic growth with pure selection (i.e., strategies
replicate themselves exactly and no mutation occurs)  model:
\begin{equation}
 \frac{d}{dt} x(t,q) = x(t,q) (
q_1 -q_2 X(t)), \label{logiseq}\end{equation}
 where $X(t) = \int_Q x(t,q) dq$ is the total population, $Q \subset \text{int}(\mathbb{R}_+^2)$ is compact
 and the state space is the set of continuous real valued functions
 $C(Q)$. Each $ q=(q_1, q_2) \in Q$ is a two tuple where $q_1$ is an
 intrinsic replication rate and $q_2$ is an intrinsic mortality
 rate. The fittest strategy was defined as the one that has the highest replication to mortality ratio, $\max_Q\{q_1/q_2\}$, ratio. The authors show that the solution converges to the Dirac mass centered at this ratio; provided there is a unique strategy with highest replication to  mortality ratio. In Figure 1 we present two examples of strategy spaces $Q$. One that
has a unique fittest strategy (left) and another that has a continuum of fittest strategies (right).

\begin{center}
{\bf Insert Figure 1 Here}
\end{center}

 The nonuniqueness property of the fittest strategy presents a mathematical difficulty that leads us to a notion of ``weak" or ``generalized"  asymptotic limit.  These weak limits live in
   a certain ``completion" of the space of finite signed measures
  and allow  for our understanding of
    the long time behavior of the pure selection model even when there is a continuum of fittest strategies. We believe this is the first time that this issue has been considered. All other
      studies including \cite{AFT} assume a unique fittest strategy.

  {In \cite{CLEVACK} we defined a evolutionary game theory model (EGT) model as an ordered triple $(Q,\mu(t),F(\mu(t)(Q)))$ subject to:
\begin{equation}\label{mconstraint}\frac{d}{dt}\mu(t)(E)=F(\mu(t)(Q))(E), \text{ for every}
~~E \in \mathcal{B}(Q). \end{equation} Here $Q$ is the strategy
(metric) space, $\mathcal{B}(Q)$ are the Borel sets on $Q$, $\mu(t)$
is a time dependent family of finite signed Borel measures on $Q$
and $F$ is a density dependent vector field such that $\mu$ and $F$
satisfy equation \eqref{mconstraint}.
 We also formulated the following EGT model as a dynamical system on the } state space of finite signed Borel measures under the  weak$^*$ topology:
\begin {equation} \left\{\begin{array}{ll}\label{M1}
 \displaystyle \frac{d}{dt}{\mu}(t; u, \gamma)(E) = \int_Q {f}_1(\mu(t)(Q), \hat q) \gamma(\hat q)(E)d\mu(t)(\hat q)\\
\hspace{1.2 in} - \displaystyle \int_E {f}_{2}(\mu(t)(Q),\hat q)
d\mu(t)(\hat q) =  {F} (\mu, \gamma)(E) \\
\mu(0; u,\gamma)=u.
\end{array}\right.\end{equation}
The purpose of this paper is to study the long term behavior of solutions to the model \eqref{M1}.

 {
This paper is organized as follows. In section 2 we mention some background definitions and information from \cite{CLEVACK} to help the reader  with a background context.  In section 3 we demonstrate that the finite signed Borel measures are asymptotically closed and provide some permanence results. In section 4 we tackle the problem of having nonunique fittest strategies. We motivate the discussion by looking at traditional ways of determining equilibria in evolutionary game theory and then we form a weak solution and study its asymptotic behavior for pure selection and small perturbation of pure selection.  In section 5, we generalize our methodology for handling pure selection to the general model. In section 6 we provide concluding remarks. Finally, for the convenience of the reader, we state known results which are needed for the development of our theory in the Appendix.}


\section{Assumptions and Background Material}
 {In this section we state assumptions that we will use throughout the paper and we recall the main well-posedness result established in \cite{CLEVACK}.
\subsection{Birth and Mortality Rates}
Concerning the birth and mortality densities $f_1$ and $f_2$ we make assumptions similar
to those used in \cite{AFT,CLEVACK}:
\begin{itemize}
\item[(A1)] $f_1: \mathbb{R}_+ \times Q \rightarrow \mathbb{R}_+$ is locally Lipschitz
continuous in $X$ uniformly with respect to $q$, nonnegative, and nonincreasing  on $\mathbb{R}_+$ in $X$
and continuous in $q$.
\item[(A2)] $f_2: \mathbb{R}_+ \times Q \rightarrow \mathbb{R}_+$ is locally Lipschitz continuous in $X$ uniformly with respect to $q$,
nonnegative, nondecreasing on $\mathbb{R}_+$ in $X$, continuous in
$q$ and $ \inf_{q \in Q} {f_{2}(0, q)} =\varpi
>0 $. (This means that there is some inherent mortality not density
related)
\end{itemize}
These assumptions are of sufficient generality to capture many nonlinearities of classical population dynamics including Ricker,
Beverton-Holt, and Logistic (e.g.,  see \cite{AFT}).}
 {
\subsection{Reproductive Numbers and Carrying Capacities}
 We define $$\mathcal{R}(X, q) =\frac{f_1( X,q)}{f_2(X,q)}$$ to be
the reproductive number at total population size $X$, i.e., it is a
measure of the average amount of newborns contributed by an individual of
characteristic $q$ during its lifetime at population density $X$. Thus,
$\mathcal{R}(0,q)$ can be interpreted as an \emph{inherent
reproductive number}. Since $Q$ is compact, $ \mathcal{R}(0,\cdot)$ achieves a maximum and
minimum. Let $\mathfrak{Q},\mathfrak{q} $ be two points where $
\mathcal{R}(0,\cdot)$ achieves these extremum values, respectively.}

 {
From the assumptions  (A1) and (A2) above, the monotonicity properties of $f_1$ and $f_2$ imply that there exists a $K : Q \to  [0, \infty]$ given by}
{
\[
 K(q) = \begin{cases}
       \inf \{ K : \mathcal{R}(K,q)\leq 1\}   & \text{~~if ~~~} \mathcal{R} (0,q) \ge 1 \\
        0 & \text{~~if~~~ } \mathcal{R} (0,q) < 1  .
        \end{cases}
\]
Intuitively  $K(q)$ is the carrying capacity of the individuals using strategy $q$. We let $K(\mathfrak{Q})=K_{\mathfrak{Q}} $ and $K(\mathfrak{q})=k_{\mathfrak{q}}$.}

{ The problem is that under assumptions (A1)-(A2) the carrying capacity could be infinite. In studying the wellposedness in \cite{CLEVACK}
 we allowed $\infty$ as a value in the definition of $K$ above because some of the physical systems that
  we wish to model could in principle become unbounded \cite{Burger}. This unbounded case could be
  viewed as the fittest having an infinite carrying capacity. However, using \eqref{logiseq} as a model
   we notice several things. If the fittest individual has an infinite carrying capacity, then any
population evolving as in \eqref{logiseq} will grow unbounded and
will approach $\infty \delta_{\mathfrak{Q} }$, but this is not a
\textbf{finite} signed measure. Thus, in order to analyze the long time behavior of these models we must restrict carrying capacities to finite values. This prompts the following additional assumption.}

\begin{itemize}
\item [(A3)] the carrying capacity of the fittest class is finite, i.e.,  $ K(\mathfrak{Q}) < \infty .$

\end{itemize}

 {
Thinking of the outcome in \eqref{logiseq}, we see that in principle
there could exist an infinite sequence of strategies and densities
$(q_n, X_n) \rightarrow (\mathfrak{Q},K_\mathfrak{Q})$. In this
scenario, it is not clear who will survive. So we assume that
increases in population size do not change the ordering of fitness.
Hence we have the following further assumptions on
$\mathcal{R}$  and $K$.
\begin{itemize}
\item[(A4)] If $ q \neq \hat q $ and $\mathcal{R}(0,q) > \mathcal{R}(0,\hat
q)$, then $ \mathcal{R}(Y,q) > \mathcal{R}(Y,\hat q)$ if $Y \ge 0$.
Likewise, if $\mathcal{R}(0,q) = \mathcal{R}(0,\hat q)$, then $
\mathcal{R}(Y,q) =\mathcal{R}(Y,\hat q)$ if $Y \ge 0$.
\end{itemize}}

 {
 For $K$ we assume the following:
 \begin{itemize}
\item [(A5)]
\begin{itemize}
\item[] If $\mathcal{R}(0,\mathfrak{Q})\ge 1$, then there is a unique $
K_{\mathfrak{Q}} \in [0, \infty) $ such that
$\mathcal{R}(K_{\mathfrak{Q}},\mathfrak{Q})=1$.
 If $\mathcal{R}(0,\mathfrak{q})\ge 1$, then there is a unique
 $ k_{\mathfrak{q}} \in [0,\infty)$, such that $1= \mathcal{R}(k_{\mathfrak{q}},\mathfrak{q})$.
\end{itemize}
\end{itemize}}

\subsection{Main Theorem from \cite{CLEVACK}}
The following is the main well-posedness theorem taken from \cite{CLEVACK}. Here $ \mathcal{M}= \mathcal{M}(Q)$  are the finite signed Borel measures on $Q$, a compact Polish space, $\mathcal{M}_{V,+}$ represents the positive cone under the total variation topology and $\mathcal{M}_{w,+}$ represents the positive cone under the weak* topology.
 Let $\mathcal P_w= \mathcal P_w(Q) $ denote the probability
measures under the weak* topology and $ C^{po} =
C(Q,\mathcal{P}_w(Q))$, the continuous $P_w$ valued functions on $Q$ with the topology of uniform convergence.

\begin{theorem}\label{main}  Assume that (A1)-(A2) hold. There exists a continuous dynamical system $({\cal M}_{w,+
}, C^{po},\varphi)$ where $\varphi:
{\mathbb {R}} \times{\cal M} _{w,+} \times C^{po} \to {\cal M}_{w,+}
  $ satisfies the following:
\begin{enumerate}
\item  For fixed $ u, \gamma $, the mapping $t \mapsto \varphi(t;
u,\gamma)$ is continuously differentiable in total variation, i.e.,
$\varphi(\cdot, u, \gamma): {\mathbb {R}} \to {\cal M}_{V,+}$.
 \item For fixed $ u, \gamma $, the mapping $t \mapsto \varphi(t; u,\gamma)$ is the unique \emph{solution} to

 \begin {equation} \left\{\begin{array}{ll}\label{M}
 \displaystyle \frac{d}{dt}{\mu}(t)(E) = \int_Q {f}_1(\mu(t)(Q), \hat q) \gamma(\hat q)(E)d\mu(t)(\hat q)\\
\hspace{1.2 in} - \displaystyle \int_E {f}_{2}(\mu(t)(Q),\hat q)
d\mu(t)(\hat q) =  {F} (\mu, \gamma)(E) \\
\mu(0)=u.
\end{array}\right.\end{equation}
\end{enumerate}
\end{theorem}

\section{Results for the Asymptotic Dynamics of the Full Model}
In this section we begin study of the long time behavior of the full model. In particular, we provide sufficiency for permanence and we show that $\mathcal{M}(Q)$ is what we will define as \textbf{asymptotically closed}.

\subsection{Permanence} {Permanence} here means
$$0< \liminf_{t \rightarrow \infty} \mu(t)(Q)\le  \limsup_{t
\rightarrow \infty} \mu(t)(Q) < \infty.$$

\begin{theorem}\label{BS}(Bounds for Solution)
Assume (A1)-(A5), then for \eqref{M} we have the following:
 \begin{equation}\label{limbound}
\min\{ k_{\mathfrak{q}}, \mu(0)(Q)\} \le \mu(t)(Q) \le \max
\{\mu(0)(Q), K_\mathfrak{Q} \}, \qquad \text{for all }~  t
~\geq 0,
  \end{equation}
 and \begin{equation}\label{limsupbound}
k_{\mathfrak{q}}  \leq \liminf_{t \to \infty} \mu (t)(Q) \leq
\limsup_{t \to \infty} \mu(t) (Q) \le K_{\mathfrak{Q}} .
\end{equation} Hence, if $k_{\mathfrak{q}} >0$
 then the population is permanent.

\end{theorem}

\begin{remark} It seems at first glance that $k_{\mathfrak{q}} >0$ or
 $( \mathcal{R}(0,\mathfrak{q}) >1)$ is too restrictive of an assumption
for proving persistence, i.e., $ \liminf_{t \rightarrow \infty}
\mu(t)(Q) >0 .$ However, if $k_{\mathfrak{q}} =0$  and $ \gamma(\hat
q) = \delta_{\mathfrak{q}}$ then it is an exercise to show that the
model converges to the zero measure even in setwise convergence.
\end{remark}

\begin{proof} We first prove the rightmost inequalities, i.e., those for $\max$ and $\limsup$.
 We will let $X(t)$ denote $ \mu(t)(Q)$ at
times in this section. First notice that
\begin{equation}\label{TOTPOPBOUND}
 \begin{array}{lll} \dot{X}(t) = \dot{\mu}(t)(\textbf{1}) &=&
\int_{Q}\Bigl[f_{1}(\mu(t)(Q),q)-f_{2}(\mu(t)(Q),q)\Bigr]d\mu(t)(
q) \\
&\leq&
 \int_{Q}\Bigl[\cR(\mu(t)(Q),\mathfrak{Q} )f_{2}(\mu(t)(Q),q)
-f_{2}(\mu(t)(Q),q)\Bigr]d\mu(t)( q)  \\
&=& \Bigl [\cR(\mu(t)(Q),\mathfrak{Q} ) -1 \Bigr]\int_Q
f_2(\mu(t)(Q), \hat q) d\mu(t)( \hat q).
\end{array}\end{equation}
 For the second inequality in \eqref{limbound} we have cases.  If $K_{\mathfrak{Q}}
=0$, then by \eqref{TOTPOPBOUND} we see that $\dot{X}(t) \leq 0$,
for all $t \geq 0 $. Now assume $0<K_{\mathfrak{Q}}  < \infty
$. Starting from \eqref{TOTPOPBOUND}, we see that if $X(t)>
K_{\mathfrak{Q}} $ then $ \dot{X}\leq 0$ therefore it follows from
basic analysis or \cite[Lemma A.6]{Thi03} that
\begin{equation}\label{limbound1} X(t)\leqslant \max\{ K_{\mathfrak{Q}},
X(0)\},
 \text { for all } t \ge~ 0 . \end{equation}

Likewise for \eqref{limsupbound} suppose $
0\le K_{\mathfrak{Q}}  < \infty $ and assume that
\[
 K_{\mathfrak{Q}}< K^\infty := \limsup_{t \to \infty} X(t).
\]
Then we claim that there is no sequence $X(t_j)$ that converges to
$K^\infty.$ Indeed if $X(t_j)$ is such a sequence, then there are
two cases. First suppose there is a $t_k$ such that
$K_{\mathfrak{Q}}  < X(t_k) < K^\infty$ for some time $t_k$. Then
$X(t_k)$ is an upper bound for $ \{X(t) : t \ge t_k\} $. Indeed if $ S =\{ t ~~|~~ t>t_k$ and  $ X(t) > X(t_k) \} \neq \emptyset $,  let
$T_{inf}= \inf S $. By continuity $ X(T_{inf})= X(t_k).$  Now for every $n$, there is a $T_n$ and $S_n$ such that
  \begin{enumerate}
  \item  $ X(T_n) > X(T_{inf})$ and $T_n \in (T_{inf},T_{\inf} + \frac{1}{n}]$
  \item   $X(T_n)-X(T_{inf})= \dot{X}(S_n)(T_n -T_{inf})$ where $T_{inf} < S_n < T_n$
  \end{enumerate}
    Hence $\dot{X}(S_n)> 0$
and by the inequalities in the second line preceding \eqref{limbound1},
 $X(S_n) \le K_{\mathfrak{Q}}.$
Continuing in this fashion we
get a sequence $\{S_{n} \}$ where $X(S_{n}) \le K_{\mathfrak{Q}}
$ and $\{S_{n} \} \rightarrow T_{inf} $. By continuity $X(T_{\inf}) \le
K_{\mathfrak{Q}}$. This is a contradiction.

Secondly, assume $ K^{\infty} \le X(t_j) $ for all $j$. Let $\xi
>0$ be such that $ K_{\mathfrak{Q}}  < \xi < K^\infty$. Then by \eqref{TOTPOPBOUND}, $\dot{X}(t_j)\le 0$
for all $j$. Hence using an argument similar to the one in the
previous paragraph and the fact that $ t_j \rightarrow \infty $ we
see that for any $T \ge t_1$, there is an $L$ such that  $ X(t_1)\ge
X(T)\ge X(t_L) \ge K^\infty > \xi$.  Then, by \eqref{TOTPOPBOUND},
$\dot{X}(T)\le [\mathcal{R}(\xi,\mathfrak{Q} ) -1]\varpi X(T) $, and
hence $X(t_j)$ decreases to $0$. This contradicts the fact that
$\inf \{X(t_j): j \in \mathbb{N}\} \ge K^{\infty} > K_{\mathfrak{Q}}
\ge 0$.

Now we prove the leftmost inequalities, i.e., those for $\min$ and $\liminf$.
Note that
\begin{equation}\label{TOTPOPBOUND2}
 \begin{array}{lll} \dot{X}(t) = \dot{\mu}(t)(\textbf{1}) &=&
\int_{Q}\Bigl[f_{1}(\mu(t)(Q),q)-f_{2}(\mu(t)(Q),q)\Bigr]d\mu(t)(
q) \\
&\geq&
 \int_{Q}\Bigl[\cR(\mu(t)(Q),\mathfrak{q} )f_{2}(\mu(t)(Q),q)
-f_{2}(\mu(t)(Q),q)\Bigr]d\mu(t)( q)  \\
&=& \Bigl [\cR(\mu(t)(Q),\mathfrak{q} ) -1 \Bigr]\int_Q
f_2(\mu(t)(Q), \hat q) d\mu(t)( \hat q).
\end{array}\end{equation}
 Once again, for the first inequality in \eqref{limbound} we have cases. The
case $k_{\mathfrak{q}} =0$ is trivial.  Now assume $0<k_{\mathfrak{q}}  <
\infty $. Starting from \eqref{TOTPOPBOUND2}, we see that if $X(t)<
k_{\mathfrak{q}} $ then $ \dot{X}\geq 0$ therefore it follows from
basic analysis or \cite[Lemma A.6]{Thi03} that
\begin{equation}\label{limbound1b} X(t)\geqslant \min\{ k_{\mathfrak{q}},
X(0)\},
 \text { for all } t \ge~ 0 . \end{equation}

Likewise for \eqref{limsupbound} we have cases. The
$k_{\mathfrak{q}} = 0$ case is trivial.  Otherwise suppose $ 0<
k_{\mathfrak{q}}  < \infty $ and assume that
\[
k_\infty := \liminf_{t \to \infty} X(t) < k_{\mathfrak{q}} .
\]
Then we claim that there is no sequence $X(t_j)$ that converges to
$k_\infty.$ Indeed if $X(t_j)$ is such a sequence, then there are
two cases. First suppose there is a $t_k$ such that $k_\infty
<X(t_k) < k_{\mathfrak{q}}$ for some time $t_k$. Then $X(t_k)$ is
a lower bound for $ \{X(t) : t \ge t_k\} $.
   Indeed if $ S =\{ t ~~|~~ t_k<t$ and  $ X(t) < X(t_k) \} \neq \emptyset $,  let
$T_{inf}= \inf S $. By continuity $ X(T_{inf})= X(t_k).$  Now for every $n$, there is a $T_n$ and $S_n$ such that
  \begin{enumerate}
  \item  $ X(T_n) < X(T_{inf})$ and $T_n \in (T_{inf},T_{inf} + \frac{1}{n}]$
  \item   $X(T_{n})-X(T_{inf})= \dot{X}(S_n)(T_{n} -T_{inf})$ where $T_{inf} < S_n < T_{n}$
  \end{enumerate}
    Hence $\dot{X}(S_n)< 0$
and by the inequalities in the second line following \eqref{TOTPOPBOUND2},
 $X(S_n) \ge k_{\mathfrak{q}}.$
Continuing in this fashion we
get a sequence $\{S_{n} \}$ where $X(S_{n}) \ge k_{\mathfrak{q}}
$ and $\{S_{n} \} \rightarrow T_{inf} $. By continuity $X(T_{inf}) \ge
k_{\mathfrak{q}}$. This is a contradiction.

Secondly, assume $X(t_j) \le k_{\infty}$ for all $j$. Let $\xi
>0$ be such that $ k_{\infty}<\xi< k_{\mathfrak{q}}$. Then by \eqref{TOTPOPBOUND2}, $\dot{X}(t_j)\ge 0$
for all $j$. Hence using an argument similar to the one in the
previous paragraph and the fact that $ t_j \rightarrow \infty $ we
see that for any $T \ge t_1$, there is an $L$ such that  $ X(t_1)\le
X(T)\le X(t_L) \le k_\infty < \xi$.  Then, by \eqref{TOTPOPBOUND2},
$\dot{X}(T)\ge [\mathcal{R}(\xi,\mathfrak{q} ) -1]\varpi X(T) $, and
hence $X(t_j)$ increases to $\infty$. This contradicts the fact that
$\sup \{X(t_j): j \in \mathbb{N}\} \le k_{\infty} < k_{\mathfrak{q}}
<\infty$.
\end{proof}

Next we prove a persistence result for the case where
$k_\mathfrak{q}$ is not necessarily positive. In order to obtain
population persistence a condition on the proportion of offspring of
strong traits (through the mutation parameter $\gamma$) is imposed.
In particular, this result states that for the population to persist
the proportion of offspring of individuals with strong traits that
belong to these strong traits must be large enough. In other words,
individuals with strong traits cannot afford to reproduce a large
portion of their offspring that carry weak traits if the population
is to survive.

\begin{theorem}\label{persistence}

Assume that (A1)-(A5) hold and that there exists some Borel set $ E \subseteq Q $, with the
following properties:

\begin{enumerate}

\item There exists an $\epsilon > 0$ such that $ \inf_{q \in E} \mathcal{R}(\epsilon,q) \gamma(q)(E) >
1$.
\item $ \mu(0)(E)>0$.
\end{enumerate} Then the population is uniformly weakly persistent, i.e., $\limsup_{t\to \infty} \mu(t)(Q) \ge \epsilon$.
\end{theorem}

\begin{proof}
Assume that $\limsup_{t\to \infty} \mu(t)(Q)< \epsilon$. Then for
sufficiently large $t$,  since $f_1$ is decreasing and $f_2$ is
increasing in the first variable and $\mu(t)(Q)$ is nonnegative,
$$\begin{array}{ll} \dot \mu (t)(E) & \ge \int_E f_1(\epsilon,q) \gamma(q) (E) d\mu(t)( q) - \int_E f_2(\epsilon,q) d\mu(t)( q)\\
& = \int_E [{\cal R} (\epsilon,q) \gamma(q)(E) -1] f_2(\epsilon,q)
d\mu(t)(q) . \end{array}$$ Thus,
$$\dot \mu(t)(E) \ge [\inf_{q\in E} {\cal R}(\epsilon,q) \gamma(q)(E) -1] \int_E f_2(\epsilon,q) d\mu(t)(q).$$
Then,
$$\dot \mu(t)(E) \ge [\inf_{q\in E} {\cal R}(\epsilon,q)\gamma(q)(E)-1] \inf_{q\in E} f_2(\epsilon,q) \mu(t)(E).$$
So $\mu(t) (E) \to \infty$ since $\mu(0)(E) > 0$ and $\inf_{q\in E}
f_2(\epsilon,q)> 0$ by (A2). Contradiction to the assumption. (The bound on $\mu$ established in the previous lemma) Thus, uniform weak persistence
follows.
\end{proof}

\begin{remark}  One can apply persistence theory \cite[pg. 440]{Thi03} to show that $\mu$ in the
above result is uniformly persistent, i.e., $\liminf_{t\to \infty}
\mu(t)(Q) \ge \epsilon$. In fact, it is enough to recall from the
well-posedness result (Theorem \ref{main}) that the semiflow is
continuous in the weak* topology. Furthermore, since the solutions
have a bounded attractor in the total variation topology (Theorem
\ref{BS}) then they have a compact attractor in the weak* topology.
\end{remark}

\subsection{Measures are Asymptotically Closed}  {Suppose $ M_0 \subseteq \mathcal{M}(Q)$ is a linear subspace,
$(Q,\mu(t),F(\mu(t)(Q)))$ is a measure valued EGT model,  where $F : M_0 \rightarrow M_0 $. Then if
$\lim_{t \rightarrow \infty}\mu(t) =\xi$ (where $\xi$ is in some weak* topological closure of $\mathcal{M}(Q)$), implies $\xi \in M_0$, then $M_0$ is called asymptotically closed with respect to $F$.}

 {
\begin{theorem}\label{ACLOSED}  If $(Q,\mu(t),F(\mu(t)(Q)))$ is as in \eqref{M1}, then $\mathcal{M}_w(Q)$ is asymptotically closed with respect to $F$.
\end{theorem}
\begin{proof}
$\mathcal{M}(Q)$ is isomorphic to the dual of $C(Q)$ as Banach spaces and hence they have the same weak* topology. However, by \cite{GARRETT} the dual of $C(Q)$ is weak* quasi-complete. Hence, the measures are quasi-complete.
By the remark following Theorem \ref{persistence}  we see that our model has a compact attractor. Hence it is asymptotically closed with respect to $F$.
\end{proof}}
 {
Note that $C(Q)$ (where we integrate continuous functions with respect to some probability measure) and $L^1(Q,\nu)$  (where $\nu$ is some finite measure) do not have this property. The measures are universal in some categorical sense {and evolutionary game theory captures this nicely}. }

\section{Asymptotic Analysis with Nonunique Fittest}

\subsection{Definition of ESS and CSS}
 When attempting to analyze the asymptotic behavior of an EGT model, one
usually first forms the appropriate notion of an \emph{Evolutionary
Stable Strategy} or \emph{ESS.} The concept of an \emph{ESS} was
introduced into biology from the field of game theory by Maynard
Smith and Price to study the behavior of animal conflicts
\cite{Maynard2}. Intuitively an\emph{ ESS} is a strategy such that
if all members of a population adopt it, no differing behavior could
invade the population under the force of natural selection. So at
the ``equilibrium" of an \emph{ESS } all other strategies, if
present in small quantities, should have negative fitness and die
out.

 The standard way to locate an \emph{ESS} is to maximize some fitness
measure \cite[chap. 7]{BrownVincent}. However this method has
glaring defects even though it is quite a natural way to find ESS's.
 Suppose our game is as in \cite{AMFH}. There $Q\subseteq
\text{int}(\mathbb{R}^2)$ and each $q=(q_1,q_2)$, where $q_1$ is an
intrinsic birthrate and $q_2$ is an intrinsic mortality rate. Also
$Q$ is a rectangle as in Figure 1 (left). So if $\mathcal{R}(0,q) =
\frac{f_1(0)q_1}{f_2(0)q_2}$ is the net reproductive number and the
strategy with the maximal net reproductive number is the fittest,
then clearly the (bottom-right) corner point of the rectangle is the fittest. It has the maximal
birth to death ratio. However, if there is a different $Q$ as in
Figure 1 (right) we see that the fittest is actually now a continuum
of points. Clearly with this new strategy space, maximizing the net
reproductive number will not yield a single strategy but a continuum
and hence no ESS. {
So we need a way to mathematically say that an entire class is the fittest. The idea is to form partitions of $Q$ and
to place the quotient topology on the partitions. Then form the measures on the elements of the partition and consider a weak EGT model on the measures of the partition elements. For example, given the set $Q$, there are the two trivial partitions. The one where every element is related only to itself,  $[Q]_{\bot}$,  and the one where every element relates every point to all other points, $[Q]_{\top}$. More specifically $ [Q]_{\bot}={ \{ [q]_{\bot}~|~ q ~ \in Q \}} $ where $[q]_{\bot}={q}$. Similarly for $[Q]_{\top}$, except $[q]_{\top}= Q$ for every $q \in Q$. There are also nontrivial partitions that are important from both a mathematical and evolutionary game theoretic viewpoints. For the remainder of this section we will focus on the following partition as it is both illustrative of the technique we are developing and informative biologically. Let $\mathcal{R}(0,q) =
\frac{f_1(0)q_1}{f_2(0)q_2}$. Clearly this is a continuous function on $Q$ and hence we can form the partition $ [Q]_{\mathcal{R}} $ (see Figure 2 for an example of such a partition) as an aid in defining an ESS and later a CSS.
Here $ [Q]_{\mathcal{R}} =\{ [q]_{\mathcal{R}}~|~ q~ \in Q \} $ and $[q]_{\mathcal{R}}= \mathcal{R}^{-1}[\mathcal{R}(0,q)]$.}

 {
Next consider the commutative diagram
\begin{equation}\label{diagram} \xymatrix{
Q=[Q]_{\bot} \ar[d]_{\phi_{\bot,\mathcal{R}}} \ar[rd]^{\mathcal{R}(0,\cdot)} & \\
[Q]_{\mathcal{R}(0,\cdot)} \ar[r]_{\overline{\mathcal{R}}} &
\mathbb{R} }\end{equation} where  $\phi_{\bot,\mathcal{R}}$ is the quotient map and $ \overline{\mathcal{R}}$ is the
induced map $ \overline{\mathcal{R}}([q]_{\mathcal{R}}) =
\mathcal{R}(0,q) $.}
 {
Clearly $\mathcal{R}(0, \cdot)$ is a continuous surjection onto its
image. Since $Q$ is compact and $\mathbb{R}$ is Hausdorff,
$\mathcal{R}(0, \cdot)$ is a quotient map (closed map). Hence
Theorem \ref{AdjointTheorem} implies that $\overline{\mathcal{R}}$
is a homeomorphism onto the image of $\mathcal{R}(0, \cdot)$. Hence
we can define a metric
\begin{equation} \rho([q]_{\mathcal{R}}, [\hat q]_{\mathcal{R}})= |\mathcal{R}(0,q)-
\mathcal{R}(0,\hat q)|.\end{equation} This makes
$([Q]_{\mathcal{R}(0, \cdot)},\rho )$ into a compact metric space,
since it is homeomorphic to the image of $\mathcal{R}(0, \cdot)$.}

 {
Since $([Q]_{\mathcal{R}(0, \cdot)},\rho )$ is a compact metric space then we can form $\mathcal{M}_w([Q]_{\mathcal{R}(0, \cdot)})$. These are the finite signed Borel measures on $([Q]_{\mathcal{R}(0, \cdot)},\rho )$ in the weak* topology. The map $\phi_{\bot,\mathcal{R}}^*[\mu(t)](E) = \mu(t)( \phi_{\bot,\mathcal{R}}^{-1}(E))$ will take a dynamical system on $\mathcal{M}_w(Q)$ to a dynamical system on $\mathcal{M}([Q]_{\mathcal{R}(0, \cdot)})_w$.
  If
$\mu(t)$ is evolving according to pure replicator dynamics on $Q$,
we can define $\mu_{\mathcal{R}}^*(t):= \phi_{\bot,\mathcal{R}}^*(\mu)(t)$, $ [q]_{\mathcal{R}} := [q]_{\mathcal{R}(0,
\cdot)}$ and $([Q]_{\mathcal{R}},\rho ):= ([Q]_{\mathcal{R}(0,
\cdot)},\rho )$. We can then view $\mu_{\mathcal{R}}^*(t)$ as a weak version of $\mu(t)$.
Since $\phi_{\bot,\mathcal{R}}$ is continuous, then using Theorem \ref{changeofvar} in the Appendix we can see that if $ \mu(t) \to \mu_{0} $
 in $\mathcal{M}_w(Q)$, then $ \phi_{\bot,\mathcal{R}}^*[\mu(t)]  \to
\phi_{\bot,\mathcal{R}}^*[\mu_{0}]$ on $\mathcal{M}_w([Q]_{\mathcal{R}}) $ i.e. their weak versions will converge as well. This is important if we wish to show that we really have a ``generalized" dynamical system. When we extend our notion of dynamical system we still keep our old notions of convergence.

 With this background we are now ready to continue our study of EGT dynamics.}

\begin{center}
{\bf Insert Figure 2 Here}
\end{center}

We use the above discussion to define ESS and CSS as follows: We define for two strategy classes $[q]_{\mathcal{R}}$ and $[\hat
q]_{\mathcal{R}} $ a relative fitness. Then using this definition we
define an \emph{ESS}. To this end define the relative fitness
between two strategy classes as
$$\lambda_R([q]_{\mathcal{R}}, [\hat q]_{\mathcal{R}}) =\mathcal{R}(K(q), \hat q) -1.$$
This is clearly well-defined and is a measure of the long term
fitness of $[\hat q]_{\mathcal{R}}$ when $[q]_{\mathcal{R}}$ is in
equilibrium.
\begin{definition} A strategy class $[q]_{\mathcal{R}}$ is a  \textbf{(local) global \emph{ESS}} if $ [q]_{\mathcal{R}}$
satisfies $$\lambda_{R}([q]_{\mathcal{R}},[\hat q]_{\mathcal{R}}) <
\lambda_{R}([q]_{\mathcal{R}},[q]_{\mathcal{R}}), \text{ for all }
[\hat q]_{\mathcal{R}} ~(\text{in a neighborhood of }
[q]_{\mathcal{R}}).  $$ In this work an ESS is global unless explicitly mentioned as local.
\end{definition}

Notice that $\lambda_{R}([q]_{\mathcal{R}},[q]_{\mathcal{R}}) =0$
for all $ [q]_{\mathcal{R}} $ and hence $[q]_{\mathcal{R}}$ is an
(local) \emph{ESS} if and only if $\lambda_{R}([q]_{\mathcal{R}}, [\hat
q]_{\mathcal{R}}) <0$ for all $[\hat q]_{\mathcal{R}} \neq
[q]_{\mathcal{R}}$ (in a neighborhood of $[q]_{\mathcal{R}})$. All
other classes (in a neighborhood) have negative fitness when $[q]_{\mathcal{R}}$ is at
equilibrium and hence die out.

The concept of an ESS is insufficient to determine the outcome of
the evolutionary game, since a strategy that is an \emph{ESS} need
not be an evolutionary attractor \cite[chpt. 6]{BrownVincent}. It is
not the case that all members of the population will end up playing
that strategy. An \emph{ESS} simply implies that if a population
adopts a certain strategy (phenotype, language or cultural norm
etc.), then no mutant small in quantity can invade or replace this
strategy class.

\begin{definition} Suppose a population is evolving according to
\eqref{M}. If
$c_{[q]_{\mathcal{R}}}\delta_{[q]_{\mathcal{R}}}$ is a (local)
global attractor of $\mu^*_{\mathcal{R}}$ for some finite number
$c_{[q]_{\mathcal{R}}}$, then we call this strategy a
\textbf{(local) global Continuously Stable Strategy} or
\textbf{\emph{CSS}}. A CSS is assumed to be global unless explicitly stated to be local.
\end{definition}
This strategy, if it exists, is the endgame of the evolutionary
process. As it is attractive and once adopted, it cannot be invaded
or replaced.



\subsection{Pure Replicator Dynamics and Small Mutation Discrete}

Let $\gamma(\hat q) =\delta_{\hat q}$ and $u\in
\mathcal{M}_+.$ Substituting these parameters in \eqref{M} one
obtains the pure selection model
\begin {equation} \left\{\begin{array}{ll}\label{selection}
 \displaystyle \frac{d}{dt}{\mu}(t;u, \gamma)(E) = \int_E \left ({f}_1(\mu(t)(Q), \hat q)-{f}_{2}(\mu(t)(Q),\hat q)
\right )d\mu(t)(\hat q) \\
\mu(0;u,\gamma)=u.
\end{array}\right.\end{equation}

In this subsection we show that $[\mathfrak{Q} ]_\mathcal{R} $ is a
$CSS$. This means that if a population is evolving according to pure
replicator dynamics \eqref{selection}, then a multiple of $\delta_{[\mathfrak{Q}
]_{\mathcal{R}}}$ is a global attractor of $\mu_{\mathcal{R}}^*(t).$

\begin{theorem}\label{CSS1}  Assume (A1)-(A5) hold then $[\mathfrak{Q} ]_\mathcal{R}$ is a $CSS$. In particular,
if the population $ \mu(t) $ is evolving according to the law of
pure replicator dynamics \eqref{selection} and $ [\mathfrak{Q}
]_\mathcal{R} \cap supp (\mu(0)) \ne \emptyset$, then
$$\mu_{\mathcal{R}}^*(t) \to K_{\mathfrak{Q}} \delta_{[\mathfrak{Q}
]_\mathcal{R}}, \quad t \to \infty
$$ in the weak$^*$ topology.
\end{theorem}
 We require two lemmas and a proposition.

\begin{lemma}\label{extinction} Assume (A1)-(A5).
Let $ U_0 $ be an open subset of $Q$ which contains $[\mathfrak{Q}
]_\mathcal{R}$. Let $\mu(t) $ be a solution evolving according to
\eqref{selection} and such that $ [\mathfrak{Q} ]_\mathcal{R} \cap
supp(\mu(0)\neq \emptyset$. Then $\mu(t) (Q \setminus U_0) \to
0 \text{ as } t \to \infty .$
\end{lemma}
\begin{proof}
 We divide this proof into two parts: I and II.

Part I: We show that if $\breve{q} \not \in [\mathfrak{Q}
]_\mathcal{R},$ then there exists $\delta = \delta(\breve{q})>0 $
such that $\mu(t) ( B_{\delta}( \breve{q})) \to 0$ as $t \to
\infty.$

\begin{enumerate}
\item Let $\mathfrak{Q}  \in supp~ (\mu(0))$. Since $[\mathfrak{Q} ]_\mathcal{R}$ and $\{\breve{q}\}$ are both closed
and metric spaces are normal there exists $\delta _1>0 $ such that $
B_{\delta_1}(\mathfrak{Q} ) \cap \overline{B_{\delta_1}( \breve{q})}
=\emptyset.$ Next choose $K > \max \{\mu(0)(Q),
K_{\mathfrak{Q}} \}.$
\item
Let $\xi = \frac{1}{\sup_{[0, K]\times Q} f_2(X,q)}$,
$\zeta=\frac{1}{\inf_{[0, K]\times Q} f_2(X,q)}$,
 $A= \xi \inf_{ [0, K]\times Q}f_2(X,q)$, and
 $B = \zeta \sup_{[0, K]\times Q} f_2(X,q)$.
 Lastly let $\eta >0 $ be such that $\mathcal{R}(X,q) < \mathcal{R}(X,\mathfrak{Q} ) - \eta
$, $ \forall q \in \overline{B_{\delta_1}( \breve{q})}.$ Such an
$\eta$ exists by continuity of $\mathcal{F}(X,q)=
\mathcal{R}(X,\mathfrak{Q}) - \mathcal{R}(X,q) > 0$ on $[0,
 K]\times \overline{B_{\delta_1}( \breve{q})}$.

\item  Let $\epsilon_2=\frac{\eta A}{2 B}$.
 Since $\mathcal{R}(X,q)$ is uniformly continuous on $[0,
 K]\times Q$, there exists $ \delta = \delta(\epsilon_2) >0$ such that if
$(X,q)$ and $(\hat X,\hat{q}) $ are within $\delta $ distance of
each other, then $\mathcal{R}(X,q), \mathcal{R}(\hat X,\hat{q})$ are
within $ \epsilon_2$ distance of each other. Let $\delta=\min \{\delta_1, \delta(
\epsilon_2) \}$, $f(X,q) = f_1(X,q)-f_2(X,q)$, $  \overline{\breve{q}} $ the point where $f(X,\cdot)$ attains its maximum value on $
\overline{B_\delta(\breve q)} $, and $\underline{\mathfrak{Q} }$
 the point where $f(X,\cdot)$ attains its minimum on $
\overline{B_\delta( \mathfrak{Q} )},$ and $\epsilon_1 = \frac{\eta A}{3}  \frac{1}{\sup_{[0, K]} |\xi
 f_2(X,\overline{\breve{q}}) - \zeta f_2(X,\underline{\mathfrak{Q} })|}. $

\item Set $x(t) = \mu(t)(B_\delta (\mathfrak{Q}))$ and
$y(t) = \mu(t) ( B_\delta (\breve q))$. By assumption, $x(0)
>0$ and hence $x(t)
>0$ for all $t \ge 0$. We now show that $y(t) \to 0$ as $t
\to \infty$. Clearly we can assume that $y(t) > 0$ for all $t \ge
0$, otherwise $y$ is identically zero and the result holds. Let
$z =y^\xi x^{-\zeta} .$ Then, denoting $\frac{d}{dt} y(t)$ by $y'$
and using similar notation for other functions, we obtain
\begin{equation}
\label{eq:quotient} \frac{z'}{z}= \xi \frac{y'}{y} - \zeta
\frac{x'}{x}.
\end{equation}

 Then we have
\[
\begin{array}{rl}
y' = & \int\limits_{ B_\delta(\breve q)} ( f_1(X,q) -  f_2(X,q))
d\mu(q) \le  f_2(X,\overline{\breve{q}})
(\mathcal{R}(X,\overline{\breve{q}})-1) y.
\end{array}
\]

Also
\[
x' = \int_{B_{\delta}(\mathfrak{Q} )} ( f_1(X,q) -f_2(X,q) ) d\mu
(q) \ge  f_2(X,\underline{\mathfrak{Q} })
(\mathcal{R}(X,\underline{\mathfrak{Q} })-1) x .\]

Since $\overline{\breve q} \in \overline{B_\delta(\breve q)}$ and $
\underline{\mathfrak{Q} } \in \overline{B_\delta( \mathfrak{Q} )}$,
we obtain
$$\begin{array}{lll} \frac{z'}{z}
&\le& \bigl [\xi f_2(X,\overline{\breve{q}})- \zeta
f_2(X,\underline{\mathfrak{Q} }) \bigr ] (\cR(X,\mathfrak{Q} )-1)  \\
&& \quad  - \eta \xi f_2(X,\overline{\breve{q}})+ \epsilon_2 \zeta f_2(X,\underline{\mathfrak{Q} })\\
 & =& I + II  ,
\end{array}$$
where
$$\begin{array}{l} I= \bigl [\xi f_2(X,\overline{\breve{q}})- \zeta
f_2(X,\underline{\mathfrak{Q} }) \bigr ](\cR(X,\mathfrak{Q} )-1) \\
II= - \eta \xi f_2(X,\overline{\breve{q}})+ \epsilon_2 \zeta
f_2(X,\underline{\mathfrak{Q} }) .
\end{array}$$

\item  Clearly $ II <  -\frac{\eta A} {2}$ and since $\epsilon_1 = \frac{\eta}{3}  \frac{A}{\sup_{[0, K]} |\xi
 f_2(X,\overline{\breve{q}}) - \zeta f_2(X,\underline{\mathfrak{Q} })|}$, $ I <
 \frac{\eta A}{3}. $

Indeed since $\epsilon_1 >0$ and $\mathcal{R}( K_{\mathfrak{Q}}
,\mathfrak{Q} ) =1 $, there is an $\varepsilon>0 $ such that
 $|\mathcal{R}(X,\mathfrak{Q} )- 1| < \epsilon_1 $ if
 $X \in (K_{\mathfrak{Q}}  - \varepsilon, K_{\mathfrak{Q}}  + \varepsilon)$.
  Since $ \limsup_{t\to \infty} \mu(t)(Q) \leq K_{\mathfrak{Q}} $,
  then  for $t$ large enough $ X(t)\in [0, \min\{K_{\mathfrak{Q}}  + \varepsilon, K \}]$.
  So let $t$ be this large. If $X(t) \le K_{\mathfrak{Q}} $,
  then $I \le 0$. Otherwise $ K_{\mathfrak{Q}}  \le X(t) \le K_{\mathfrak{Q}}
   + \varepsilon$ and $ I =|I| \le \bigl |\xi
f_2(X,\overline{\breve{q}})- \zeta f_2(X,\underline{q_{*}}) \bigr
|\epsilon_1  <\frac{\eta A}{3}$. Thus,
 $$
\frac{z'}{z}\le I+II <  -\frac{\eta A}{2} + \frac{\eta A}{3} \le
-\frac{\eta A}{6} $$ and $z(t) \to 0$ as $t \to \infty$. Since $x$
is bounded, $y(t) \to 0 $ as $t \to \infty$.
\end{enumerate}

Part II :  $Q_0 = Q \setminus U_0$ is a compact set and $Q_0
\subseteq Q .$ Since $Q_0$ is compact, there is a finite subset
$\breve Q$ of $Q_0$ such that $ Q_0$ is contained in the union of
$B_\delta (q)$, $ q \in \breve Q$. Let $y_q (t) = \mu(t) ( B_\delta
(q))$. By the claim above, $y_q(t) \to 0 $ as $t \to \infty$ for
each $q \in \breve Q$. Since $\breve Q$ is finite and $\mu(t)
(\cdot)$ is a measure,
\[
\mu (t) (Q_0) \le \sum_{q \in \breve Q} y_q(t) \to 0 , \qquad t \to
\infty.
\]

\end{proof}

\begin{lemma}
\label{populationlimit} Assume (A1)-(A5). If $\mu(t)$ is a solution
evolving according to \eqref{selection} and such that $[\mathfrak{Q}
]_\mathcal{R} \cap supp( \mu(0)) \neq \emptyset$, then
$$ \lim_{t \to \infty} \mu(t) (Q) =K_{\mathfrak{Q}} .$$
\end{lemma}
\begin{proof} If $K_{\mathfrak{Q}} = 0$, then from \eqref{limsupbound} we
get $$\lim_{t \to \infty} \mu(t) (Q) = 0,$$ and the result follows.

Next suppose $ 0< K_{\mathfrak{Q}}  < \infty $.
 We know from Theorem \ref{BS} that $$ \limsup_{t\to \infty } \mu(t) (Q) \leq
K_{\mathfrak{Q}} .$$ Suppose $\alpha^\infty = \limsup_{t\to \infty }
\mu(t) (Q) < K_{\mathfrak{Q}} $, then we have the following.

\begin{enumerate}
\item By definition of $K_{\mathfrak{Q}} $, there exists a $\overline{K}$ such that $\alpha^\infty <\overline{ K} < K_{\mathfrak{Q}} $ and
 for all $\mathfrak{Q}  \in [\mathfrak{Q} ]_\mathcal{R} $,
$ 1 =\cR(K_{\mathfrak{Q}} ,\mathfrak{Q} ) <
\cR(\overline{K},\mathfrak{Q} )\leq \cR(\alpha^\infty, \mathfrak{Q} )$,
i.e., for all $\mathfrak{Q}  \in [\mathfrak{Q} ]_\mathcal{R} $, $
\cR(\overline{K},\mathfrak{Q} ) > 1. $

\item Since $\mathcal{R}(\overline{K},\cdot)$ is continuous and $[\mathfrak{Q} ]_\mathcal{R}$ is compact,
 there
 exists a $\underline{\mathfrak{Q}}$ where $\cR(\overline{K},\cdot) $ has a minimal value on the compact
 set $[\mathfrak{Q} ]_\mathcal{R}$. Using 1. we can find $\xi$ such that $ 1 < \xi < \cR(\overline{K},\underline{\mathfrak{Q}}). $

\item Since $Q$ is compact $\cR(\overline{K},\cdot)
 $ is uniformly continuous. So if $ 0 < \epsilon <
 \frac{\cR(\overline{K},\underline{\mathfrak{Q}})- \xi}{2} $ there exists some $\delta
> 0$ such that for all $q$, $ \hat q \in Q $, the distance between
$(\cR(\overline{K},q)$, $\cR(\overline{K},\hat q ))$ is less than
 $\epsilon $, provided that the distance between $(\overline{K},q)$,$(\overline{K},\hat q)$ is
smaller than $\delta. $
 Hence, $ 1 < \xi < \cR(\overline{K},q) $ on $U =
U_\delta ([\mathfrak{Q}]_\mathcal{R} ) = \{q \in Q :
d(q,[\mathfrak{Q}]_\mathcal{R}  )< \delta \}$, where
$d(q,[\mathfrak{Q}]_\mathcal{R} )$ is the distance from the point
$q$ to the set $[\mathfrak{Q}]_\mathcal{R} $.

\item Since $\mu(0)(U) > 0$, this implies $\mu(t)(U) > 0$ for all $t
\ge 0$. By assumption, there exists some $r > 0$ such that
$X(t)= \mu(t)(Q) \le \overline{K}$ for all $t \ge r$. If $X(t) = \mu(t)(Q)$ and
$t \ge r$, we have
\[
\begin{array}{rl}
\frac{d}{dt}\mu(t) (U) = & \displaystyle \int_U [ \cR(X(t),q) - 1]
f_2(X(t),q) d\mu (t)(q)
\\
\ge & \displaystyle \int_U [ \cR(\overline{K},q) - 1]
f_2(X(t),q) d\mu (t)(q)\\
\ge & \displaystyle \int_U \Bigl [\inf_{ \mathfrak{Q}  \in
[\mathfrak{Q}]_\mathcal{R} } \cR(\overline{K},\mathfrak{Q} )
-\epsilon -1\Bigr] f_2(X(t),q) d\mu (t)(q)
\\
= & \displaystyle \Bigl [\inf_{ \mathfrak{Q}  \in
[\mathfrak{Q}]_\mathcal{R} } \cR(\overline{K},\mathfrak{Q} )
-\epsilon -1\Bigr]\int_U f_2(X(t),q)d\mu (t)(q)
\\
 \ge & \displaystyle
\bigl[\cR(\overline{K},\underline{\mathfrak{Q}}) -\epsilon -1 \bigr]
\inf _{\overline{U}} f_2(0,q) \mu(t)(U).
\end{array}
\]

Since $\bigl[\cR(\overline{K},\underline{\mathfrak{Q}}) -\epsilon -1
\bigr] \inf _{U} f_2(0,q) > 0 $, $\mu(t)(U)$ increases exponentially
and hence $X(t) = \mu(t)(Q)$ grows unbounded. This is a
contradiction.

This implies that $\limsup_{t \to \infty } \mu(t) (Q) \ge
K_{\mathfrak{Q}}  > 0$. Combined with the  statement $\limsup_{t\to
\infty} \mu(t) (Q) \le
 K_{\mathfrak{Q}} $ from Theorem \ref{BS}, this yields $$\limsup_{t \to \infty } \mu(t) (Q) =  K_{\mathfrak{Q}}  >
 0.$$

\item Since $ K_{\mathfrak{Q}}  =\limsup_{t\to \infty } \mu(t) (Q)$, it will suffice
to show $\liminf_{t \to \infty} \mu(t) (Q) = K_{\mathfrak{Q}} .$ If
$U = U_\delta ([\mathfrak{Q} ]_\mathcal{R}) = \{q \in Q :
d(q,[\mathfrak{Q} ]_\mathcal{R} )< \delta \}$, where
$d(q,[\mathfrak{Q} ]_\mathcal{R})$ is the distance from the point
$q$ to the set $[\mathfrak{Q} ]_\mathcal{R}$, then $U =
U_\delta([\mathfrak{Q} ]_\mathcal{R})$ is a relatively open subset
of $Q$ which contains $[\mathfrak{Q} ]_\mathcal{R}$. Set $U^c = Q
\setminus U$, then $\mu(t)(Q) =\mu(t) (U) + \mu(t)(U^c)$.
 By Lemma
\ref{extinction}, $\mu(t) (U^c) \to 0$, as $t \to \infty$.
 So we only need to find a suitable $\delta$ and derive a contradiction from the assumption
 \[
 \liminf_{t \to \infty} \mu(t) (U) = \alpha_\infty < K_{\mathfrak{Q}}  .\]

 \item This initial argument is
strictly analogous to the above argument. Suppose that
$\alpha_\infty:= \liminf_{t\to \infty} \mu(t)(Q) < K_{\mathfrak{Q}}
\le \alpha^\infty$. Pick $\overline{K} \in (\alpha_\infty,
K_{\mathfrak{Q}} )$, $\epsilon$, $\delta$,
$\underline{\mathfrak{Q}}$ exactly as above. It follows from one
version of the fluctuation method (\cite{HHG}, see also \cite[Lemma
A.20]{Thi03}) that there exists a sequence $(t_j)$ such that $t_j
\to \infty$, $\mu(t_j) (U) \to \alpha_\infty $, for $j \to \infty $,
and $\frac{d}{dt}\mu(t_j) (U) = 0$ for all $j$. Since $\alpha_\infty
< \overline{K}$, if $j$ is large enough, $X(t_j) < \overline{K} $
and just like above we have the following estimates:
\[
\begin{array}{rl}
0=\frac{d}{dt}\mu(t_j) (U) = & \displaystyle \int_U [ \cR(X(t_j),q)
- 1] f_2(X(t_j),q) d\mu (t_j)(q)
\\
\ge & \displaystyle \int_U [ \cR(\overline{K},q) - 1]
f_2(X(t_j),q) d\mu (t_j)(q)\\
\ge & \displaystyle \int_U \Bigl [\inf_{ \mathfrak{Q}  \in
[\mathfrak{Q}]_\mathcal{R} } \cR(\overline{K},\mathfrak{Q} ) -\epsilon -1\Bigr]
f_2(X(t_j),q) d\mu (t_j)(q)
\\
= & \displaystyle \Bigl [\inf_{ \mathfrak{Q}  \in
[\mathfrak{Q}]_\mathcal{R} } \cR(\overline{K},\mathfrak{Q} ) -\epsilon
-1\Bigr]\int_U f_2(X(t_j),q)d\mu (t_j)(q)
\\
 \ge & \displaystyle
\bigl[\cR(\overline{K},\underline{\mathfrak{Q}}) -\epsilon -1 \bigr] \inf
_{\overline{U}} f_2(0,q) \mu(t_j)(U),
\end{array}
\]
 and for large $j$ we obtain the contradiction $0
> 0 $, since $\mu(t_j)(U)> 0$. This contradiction implies that \[
 \liminf_{t \to \infty} \mu(t) (U) = \alpha_\infty \ge K_{\mathfrak{Q}}  = \limsup_{t \to \infty } \mu(t) (Q),\]
and our result is immediate, i.e., $\mu(t) (Q) \to K_{\mathfrak{Q}}
$ as $t \to \infty$.
\end{enumerate}
\end{proof}

\begin{proposition}
\label{adjointprop} Assume (A1)-(A5). Let $\mu (t) $ be a solution
to \eqref{selection}:

\begin{enumerate}
\item {If $[\mathfrak{Q} ]_\mathcal{R} \cap
supp( \mu(0)) \neq \emptyset  $ and
$[\breve{q}]_{\mathcal{R}} \neq [\mathfrak{Q} ]_\mathcal{R}$, then there exists $\delta >0$, such that $\mu^{*}_{\mathcal{R}}(t)( B_{\delta}(
[\breve{q}]_{\mathcal{R}})) \rightarrow 0, \quad \text{ as } \quad t
\rightarrow \infty .$
\item If $[\mathfrak{Q} ]_\top \cap
supp( \mu(0)) \neq \emptyset, $ then we have $$\lim_{t \to \infty} \mu_{\top}^*(t)= K_{\mathfrak{Q}}\delta_{[q]_{\top}}, \text { for any } q \in Q.$$}

\end{enumerate}
\end{proposition}

\begin{proof}

\begin{enumerate}
\item Let $\delta>0 $ be such that $B_{\delta}( [\breve{q}]_{\mathcal{R}}) \cap
B_{\delta}( [\mathfrak{Q} ]_\mathcal{R}) = \emptyset $. Then $
\mu_{\mathcal{R}}^{*}(t)( B_{\frac{\delta}{2}}(
[\breve{q}]_\mathcal{R})) \rightarrow 0$. Indeed if $ V =\{ q : |
\mathcal{R}(0, q) -\mathcal{R}(0, \breve{q}| \le \frac{\delta}{2}
\}$, then $V$ is closed and $V\cap [\mathfrak{Q} ]_\mathcal{R} =
\emptyset$. Hence these two sets can be separated by open sets in
$(Q,d)$. Using Lemma \ref{extinction} our result is immediate once
we notice that :
$$ \mu_{\mathcal{R}}^{*}(t)( B_{\frac{\delta}{2}}( [\breve{q}]_\mathcal{R})) =
\mu(t) (\{q : |\mathcal{R}(0,q) -\mathcal{R}(0,\breve{q})| <
\frac{\delta}{2} \}) \leq \mu(t)(V) \rightarrow 0.$$

\item This is immediate from Lemma \ref{populationlimit}, since $$ \phi_{\top,\bot}^*\mu(t)([Q]_{\top})
 = \mu(t)(\phi_{\top,\bot}^{-1} ([Q]_{\top})) =\mu(t) (Q)\rightarrow K_{\mathfrak{Q}}. $$

\end{enumerate}

\end{proof}

\subsection{Proof of Theorem \ref{CSS1}}

\begin{proof}
We need to show that for every $f\in C([Q]_\mathcal{R}),$
\[
\int_{[Q]_\mathcal{R}} f([q]_\mathcal{R}) d\mu^{*}_{\mathcal{R}}(t)
([q]_\mathcal{R}) \to \int_{[Q]_\mathcal{R}} f([q]_\mathcal{R})
dK_{\mathfrak{Q}} \delta_{[\mathfrak{Q} ]_\mathcal{R}}
=f([\mathfrak{Q} ]_\mathcal{R}) K_{\mathfrak{Q}}  , ~\text {as
}\quad t \to \infty .
\]By Lemma \ref{populationlimit}, it is
sufficient to show that
\[
 \int_{[Q]_\mathcal{R}} f([q]_\mathcal{R}) d\mu_{\mathcal{R}}^*(t) ([q]_\mathcal{R})
  - f([\mathfrak{Q} ]_\mathcal{R}) \mu^{*}_{\mathcal{R}}(t)([Q]_\mathcal{R}) \to 0, \quad t \to \infty
.
\]
Let $\epsilon > 0$ and define $U = \{ [q]_\mathcal{R}\in
[Q]_\mathcal{R}: |f([q]_\mathcal{R}) - f([\mathfrak{Q}
]_\mathcal{R})| < \epsilon \}$. Then $U$ is an open subset of
$[Q]_\mathcal{R}$ which contains $[\mathfrak{Q} ]_\mathcal{R}$.
Using an argument similar to Part II in Lemma \ref{extinction} and
Proposition \ref{adjointprop} 1.,
$\mu^{*}_{\mathcal{R}}(t)([Q]_\mathcal{R} \setminus U) \to 0$ as $t
\to \infty$. Now
\[
\begin{array}{lll}
\Bigl | \int_{[Q]_\mathcal{R}} f([q]_\mathcal{R})
d\mu_{\mathcal{R}}^*(t) ([q]_\mathcal{R}) - f([\mathfrak{Q}
]_\mathcal{R}) \mu_{\mathcal{R}}^*(t)([Q]_\mathcal{R}) \Bigr |
&\\\le \int_U | f([q]_\mathcal{R}) - f([\mathfrak{Q} ]_\mathcal{R})|
d\mu_{\mathcal{R}}^*(t)([q]_\mathcal{R})\\
\quad  + \int_{[Q]_\mathcal{R} \setminus U} |f([q]_\mathcal{R})|
d\mu_{\mathcal{R}}^*(t)([q]_\mathcal{R}) + |f([\mathfrak{Q}
]_\mathcal{R})| \mu_{\mathcal{R}}^*(t) ([Q]_\mathcal{R} \setminus U)
\\
\le \epsilon \mu_{\mathcal{R}}^* (t) ([Q]_\mathcal{R})  + \bigl (
\sup |f| + |f([\mathfrak{Q} ]_\mathcal{R})| \bigr )
\mu_{\mathcal{R}}^*(t) ([Q]_\mathcal{R}\setminus U).
\end{array}
\]
Since $\mu_{\mathcal{R}}^*(t) ([Q]_\mathcal{R} \setminus U) \to 0$
as $t \to \infty$ and $\limsup_{t\to \infty}
\mu_{\mathcal{R}}^*(t)([Q]_{\mathcal{R}}) \le K_{\mathfrak{Q}} $, by
Proposition \ref{adjointprop} we have:
\[
\limsup_{t \to \infty } \Bigl | \int_Q f(q) \mu_{\mathcal{R}}^*(t)
([q]_\mathcal{R}) - f([\mathfrak{Q} ]_\mathcal{R})
\mu_{\mathcal{R}}^*(t)([Q]_\mathcal{R}) \Bigr | \le \epsilon
K_{\mathfrak{Q}} .
\]
Since this holds for every $\epsilon > 0$, the limit superior is 0
and the assertion follows.
\end{proof}

\begin{theorem}
\label{re:convergence}For $[Q]_{\bot} $ we have the following. Assume
(A1)-(A5). If $\mu(t;u,\delta_{\hat q})$ is a solution evolving
according to \eqref{selection} and such that $[\mathfrak{Q}
]_{\mathcal{R}} \cap supp( \mu(0)) \neq \emptyset$, then for
every $u$ there exists $\mu_{\infty}(u) \in \mathcal{M}_{w}([Q]_{\bot})$ which satisfies the
following properties:
\begin{enumerate}
\item $\mu(t;u, \delta_{\hat q}) \to \mu_{\infty}(u),\text{ as } \; t \to \infty $ in the weak$^*$
topology,
\item $supp (\mu_{\infty}(u)) \subseteq [\mathfrak{Q}]_\mathcal{R} \cap supp(u).$
\item If $[\mathfrak{Q}]_\mathcal{R}=\{\mathfrak{Q}\}$, then
$\mu(t) \to K_{\mathfrak{Q}} \delta_{\mathfrak{Q}}, \; t \to \infty
$ in the weak$^*$ topology.
\end{enumerate}

\end{theorem}

\begin{proof}
\begin{enumerate}

\item First notice if $ F \subseteq [\mathfrak{Q}]_\mathcal{R}$ then
\begin{eqnarray}
 \frac{d}{dt}{\mu}(t; u)(F)&=\int_F \left
({f}_1(\mu(t)(Q), q)-{f}_{2}(\mu(t)(Q),\hat q) \right )d\mu(t)(\hat
q)     \nonumber \\
   &=(\mathcal{R}(X,\mathfrak{Q})- 1)  \int_Ff_2(X,q)d\mu  . \label{eq:equilibria}
   \end{eqnarray}
Also,
\begin{equation}
  \dot{X}(t)  = \int_Q (f_1(X,q)- f_2(X,q))
  d\mu(t).\label{eq:equilibria2}
\end{equation}

Using arguments similar to those in the proof of Theorem \ref{BS}
together with equations \eqref{eq:equilibria} and
\eqref{eq:equilibria2} we can show that $\mu(t)(F)$ is eventually
bounded and monotone and hence converges.

 So if $E \in \mathcal{B}(Q)$
we define
\begin{equation}
\begin{array}{ll}
\mu_{\infty}(u)(E) & = \lim_{t\to \infty} \mu(t)(E) \\
& = \lim_{t\to \infty} \mu(t)(E \cap [\mathfrak{Q}]_{\mathcal{R}}) +
\lim_{t\to \infty} \mu(t)( E \cap [\mathfrak{Q}]_{\mathcal{R}}^c).
\end{array}
\label{eq:xd3}
\end{equation}We now notice that $\lim_{t\to
\infty} \mu(t)( E \cap [\mathfrak{Q}]_{\mathcal{R}}^c)=0.$  Indeed,
first notice that for every $n$,
$$ \mu(t)(Q)= \mu(t)[B_{\frac{1}{n}}([\mathfrak{Q}]_\mathcal{R})] +
\mu(t)[\left (B_{\frac{1}{n}}([\mathfrak{Q}]_{\mathcal{R}})\right ) ^c] .$$ Using the fact that
for any $\epsilon >0,$ $\lim_{t\to \infty} \mu(t)(
(B_\epsilon([\mathfrak{Q}]_{\mathcal{R}})^c) =0$ and Lemma \ref{populationlimit}
we see that for every $n$, $ K_{\mathfrak{Q}}
=\lim_{t \rightarrow \infty} \mu(t)(B_{\frac{1}{n}}([\mathfrak{Q}]_{\mathcal{R}}))$.

Since $ B_{\frac{1}{n+1}}([\mathfrak{Q}]_{\mathcal{R}}) \subseteq
B_{\frac{1}{n}}([\mathfrak{Q}]_{\mathcal{R}})$ and
 $[\mathfrak{Q}]_{\mathcal{R}} =\cap_{n=1}^\infty B_{\frac{1}{n}}([\mathfrak{Q}]_{\mathcal{R}})$ we have
$$\mu(t)([\mathfrak{Q}]_{\mathcal{R}}) = \mu(t)\bigl( \cap_{n=1}^\infty B_{\frac{1}{n}}([\mathfrak{Q}]_{\mathcal{R}})\bigr )
=\lim_{n \to \infty} \mu(t)(B_{\frac{1}{n}}([\mathfrak{Q}]_{\mathcal{R}})).$$
Hence, taking limits as $t \to \infty$  on both sides we get that
$K_{\mathfrak{Q}} = \lim_{t \to \infty} \mu(t)([\mathfrak{Q}]_{\mathcal{R}}).$ Since
$$ \mu(t)(Q)= \mu(t)([\mathfrak{Q}]_{\mathcal{R}}) +
\mu(t)([\mathfrak{Q}]_{\mathcal{R}}^c), $$ we have that  $\lim_{t
\to \infty}\mu(t)([\mathfrak{Q}]_{\mathcal{R}}^c) =0.$
 Hence,
$$ \mu_{\infty}(u)(E)= \lim_{t \to \infty} \mu(t)(E \cap [\mathfrak{Q}]_{\mathcal{R}}).$$

It then follows from \cite[pg.270 ]{Royden} that $\mu_{\infty}(u)$
as defined is a finite signed Borel measure which satisfies property
1 since \eqref{eq:equilibria} and \eqref{eq:equilibria2} imply that
we have setwise convergence (weak convergence) which is stronger
than weak$^*$ convergence.

\item This follows easily from the integral representation formula for
pure selection and Theorem \ref{CSS1}.

\item This follows easily from properties 1 and 2 since $ [\mathfrak{Q}]_{\mathcal{R}}
$ consists of a single point and Lemma \ref{populationlimit} says
that $\mu(t) (Q)\rightarrow K_{\mathfrak{Q}}.$
\end{enumerate}
\end{proof}

\subsection{Small Mutation of Discrete Pure Replicator Dynamics}
The results of this section are for discrete systems. This means
that for the pair $(u,\gamma)$, for every $\hat q$, the support of
$\gamma(\hat q)$ along with the support of $u$ is a fixed finite
set. This means that the model will always be supported on this
finite set. Mathematically studying the dynamics of the above
mentioned is equivalent to studying the dynamics of an EGT model on
a finite Polish space. Since the latter is technically easier to
handle we make the assumption that our strategy space is finite.
With this said, the main result of this section is that we
demonstrate that there is a neighborhood around the pure replicator
kernel where unique $CSS's$ are obtained.

\subsubsection {Discrete System}
To establish our results we will make use of Theorem \ref{pert1} in the appendix and the following additional assumption:
\begin{itemize}
\item [(A6)] $f_1$ and $f_2$ are $C^1$ in $X$ and $f_{1,X}(K_{\mathfrak{Q}},\mathfrak{Q})- f_{2,X}(K_{\mathfrak{Q}},\mathfrak{Q}) < 0 $
\end{itemize}

Furthermore, we assume that   $k_{\mathfrak{q}} > 0$  and
$K_\mathfrak{Q} < \infty $. Thus, by Theorem \ref{BS} the total
population is permanent.

 \begin{remark}In Theorem \ref{pert1} the assumption that $x_0$ is an interior point of $U$
 is  unnecessarily restrictive. One can use one-sided derivatives
 with respect to some cone or wedge \cite{SmithWalt}. If $Q$ is a finite set, which we can assume is $\{e_{i}\}_{i=1}^{N}$, where ${e_i}$
 is the standard unit vector in $\mathbb{R}^N$,
  then $C^{po} =\{ A |~ A \text{ is an } N\times N$ matrix whose rows sum to
  one$\}$, and $\mathcal{M}= span\{\delta_{e_{i}}\}$. In this scenario the
  equilibrium point for the dynamical system \eqref{M} is $(K_{\mathfrak{Q}},0,
  ... ,0)$ a point on the boundary of $\mathbb{R}^N_+$. So we use
  directional derivatives in the direction of the positive cone for $D_xf(x,\lambda)$.
 \end{remark}

\begin{theorem}\label{re:main} Assume (A1)-(A6) hold, $[\mathfrak{Q}]_\mathcal{R}=\{\mathfrak{Q}\}$ and $Q = \widehat{Q}\cup
[\mathfrak{Q}]_\mathcal{R}$ where $\widehat{Q}$ is finite. Let $\delta_{\hat q}$ denote the kernel $ \gamma( \hat q) = \delta_{\hat q} $ i.e. the pure selection kernel where every strategy only gives birth to its own kind and let $\mathcal{U}=
\{ \mu \in \cal{M}|$ $ \mathfrak{Q} \in \text{supp} (\mu) \} $. Then there
exists a neighborhood $U( \delta_{\hat q})$, such that for each
$\gamma \in U(\delta_{\hat q})$, there exists
 $\widehat{\mu}(\gamma)\in {\cal M}_+ $, satisfying
$F(\widehat{\mu}(\gamma), \gamma) = \vec{\textbf{0}}$ ($
\widehat{\mu}(\gamma)$ is an equilibrium point of the system).
Furthermore, $\varphi(t;u,\gamma) \rightarrow
\widehat{\mu}(\gamma)$ as $ t \rightarrow \infty $, where $ u \in
\mathcal{U}_+ $ (the positive cone of $\mathcal{U}$).

\end{theorem}
\begin{proof}
Let  $f (X, \hat q) ={f}_{1}(X, \hat q)- f_{2}(X,\hat q)$.
 Let $ F: U_+ \times C^{po}\rightarrow \mathcal{M} $ be as in
\eqref{M}. Note that when $\gamma = \delta_{\hat q}, \; \hat q \in
Q$, then  $K_{\mathfrak{Q}}\delta_{\mathfrak{Q}}$  is an equilibrium
point and from Theorem \ref{re:convergence} it is
 globally attractive for initial measures in $U_+$. Furthermore, the derivative
 $D_{\mu}F$ : $ U_+\times\C^{po} \rightarrow L(\mathcal{M},\mathcal{M})$ given by
$$\begin{array}{lll} D_{\mu}F( \mu,\gamma)(\nu)(E)& = &\bigl[\int_Q
{f}_{1,\mu}(\mu(\textbf{1}), \hat q) \gamma(\hat q)(E)d\mu(\hat
q)- \int_E {f}_{2,\mu}(\mu(\textbf{1}),\hat q)d\mu(\hat
q)\bigr]\nu(\textbf{1}) \\
&& \quad  + \int_Q {f}_1(\mu(\textbf{1}), \hat q) \gamma(\hat
q)(E)d\nu(\hat q) -\int_E {f}_{2}(\mu(\textbf{1}),\hat q)d\nu(\hat
q),\end{array}$$ is continuous.

Evaluating this operator at the equilibrium  $K_{\mathfrak{Q}}
\delta_{\mathfrak{Q}}$  we get
$$ \begin{array}{l} D_{\mu}F(
K_{\mathfrak{Q}}\delta_{\mathfrak{Q}},\delta_{\widehat{q}})(\nu)(E)=K_{\mathfrak{Q}}\bigl
[ \int_E \bigl( {f}_{1,\mu}(K_{\mathfrak{Q}}, \hat q)-
f_{2,\mu}(K_{\mathfrak{Q}},\hat q)\bigr )
d\delta_{\mathfrak{Q}}\bigr ] \nu(\textbf{1}) \\
\hspace{1.7 in}  + \int_E \bigl [{f}_1(K_{\mathfrak{Q}}, \hat q)-
{f}_{2}(K_{\mathfrak{Q}},\hat q)\bigr ] d\nu(\hat q) .\end{array} $$

  If $\lambda \notin \{f(K_{\mathfrak{Q}},\hat
 q)\}_{ \hat q \in \hat Q}
\cup \{f_{X}(K_{\mathfrak{Q}},\mathfrak{Q}) \} $, then
\begin{equation*}  \bigl(\lambda -
D_{\mu}F(K_{\mathfrak{Q}}\delta_{\mathfrak{Q}},\delta_{\widehat{q}})\bigr)^{-1}\nu(E)=
\begin{cases}
\int_{E}c(\nu,\lambda)d\delta_{\mathfrak{Q}} + \int_{E}\frac{1}{\lambda -f(K_{\mathfrak{Q}},\hat q)} d\nu(\hat q) &  \lambda \neq 0,\\
~\\
 \int_{E}\Bigl [\frac{\nu(\mathfrak{Q})}{-f_{X}(K_{\mathfrak{Q}},\hat q)} + \int_{Q}
\frac{1}{-f(K_{\mathfrak{Q}},\hat q )}d\widehat{\nu}\Bigr ]d\delta_{\mathfrak{Q}} \\
\qquad + \int_{E}\frac{1}{-f(K_{\mathfrak{Q}},\hat q)}
d\widehat{\nu}(\hat q) & \lambda =0,
\end{cases}
\end{equation*}
exists and is continuous (here $c(\nu, \lambda) =\frac{f_{X}(
K_{\mathfrak{Q}}, \mathfrak{Q})}{ \lambda -f_X( K_{\mathfrak{Q}},
\mathfrak{Q})} \int_{Q}\frac{d\nu(\hat q)}{\lambda -f(
K_{\mathfrak{Q}}, \hat q)} $ and $\widehat{\nu}= \nu-
\nu(\{\mathfrak{Q}\})\delta_{\mathfrak{Q}}$). Likewise $$ \{
\Bigl(\frac{-f_{X}(K_{\mathfrak{Q}},\mathfrak{Q})}{f_{X}(K_{\mathfrak{Q}},
\mathfrak{Q})-f(K_{\mathfrak{Q}},q)}\delta_{\mathfrak{Q}}
+ \delta_q,f(K_{\mathfrak{Q}},q) \Bigr) \}_{q \in \hat Q}, \; \Bigl(
f_{X}(K_{\mathfrak{Q}},\mathfrak{Q})\delta_\mathfrak{Q},
f_{X}(K_{\mathfrak{Q}},\mathfrak{Q})\Bigr)$$ form
eigenvector-eigenvalue pairs. Hence, the spectrum $$\sigma
(D_{\mu}F(
K_{\mathfrak{Q}}\delta_{\mathfrak{Q}},\delta_{\widehat{q}}))=\{f(K_{\mathfrak{Q}},\hat
 q)\}_{ \hat q \in \widehat{Q}}
\cup \{f_{X}(K_{\mathfrak{Q}},\mathfrak{Q}) \} $$  has negative
growth bound since each of the eigenvalues are negative. Thus, the
result follows from Theorem \ref{pert1}.
\end{proof}
\begin{remark}There is an infinite dimensional version of Theorem 8.8. In it $ U(t)\equiv  D_{\mu}F(
K_{\mathfrak{Q}}\delta_{\mathfrak{Q}},t,\delta_{\widehat{q}}) $, which defines a strongly continuous semigroup, must have a negative growth bound $( r(U(t)) = e^{-\omega t} $ with $\omega >0 ) $ \cite{SmithWalt}. This negative growth bound is needed in order to generate a contractive mapping. We cannot establish this hypothesis for our model in the infinite dimensional case. In the infinite dimensional case $U$ has $0$ in its continuous spectrum. This is due to the fact that $f(K_{\mathfrak{Q}},q)$ is an eigenvalue for every $q$. Due to continuity of $f$, $0$ is in the continuous spectrum.

\end{remark}

 \begin{corollary}
\label{re:finite} Assume that $ Q=\{q_i\}_{i=1}^N$ is a discrete set
and let $x_i(t) =\mu(t)({q_i})$, $f_{n,j}(X) = f_n(X, q_j)$ for n =
1, 2. Then the  system \eqref{M} reduces to the following
differential equations system:

\begin {equation} \left\{\begin{array}{ll}
  \frac{d}{dt}{x_i}(t;u, \lambda) =  \sum_{j=1}^N  f_{1,j}(X(t))p_{ij}x_j(t) - f_{2,i}(X(t))x_i(t)~~~~~~~i = 1, . . . , N \\
x_i(0;u,\lambda) = u_i.
\end{array}\right.\label{discrete}\end{equation}
Moreover, suppose that $P^\epsilon = p^\epsilon_{ i,j}$, that $
||P^{\epsilon} -I || \rightarrow 0 $ as $\epsilon \rightarrow 0 $
and that $f_{n,\cdot}$, n = 1, 2 is continuously differentiable.
Then for $\epsilon$ small enough there exists an equilibrium
$\overline{x}(\epsilon)$ of the ordinary differential equation system \eqref{discrete} with
$\overline{x}(\epsilon)$ converging to $\overline{x}(0) =
(K_{\mathfrak{Q}}, 0, 0, \cdots, 0)^T$ as $ \epsilon \rightarrow 0$.
Furthermore, $\overline{x}(\epsilon)$ is globally asymptotically
stable.
\end{corollary}

\section{Towards a General Theory of  Nonunique Fittest}

\subsection{Completions of Measure Spaces and a Fundamental Theorem } \label{AN}
 {
Up to section 4 we have a well-posed model with the property that if
 \begin{equation}\label{weakform} \lim_{t \rightarrow
\infty} \mu(t) = \nu ,\end{equation}
then $\nu \in \mathcal{M}_w(Q)$.  However if there is more than one fittest strategy, this solution $\nu$ will quickly cease to be a globally attracting equilibrium. However, the result from \cite{AFT} intimates that even though the strategy itself ceases to be a globally attracting equilibrium the set of all fittest strategies is a globally attracting equilibrium ``in some sense". This section is a mathematical framework in which we can make sense of the class of fittest strategies being a globally attracting equilibrium.}
 {
So this section is devoted to completing measure spaces in such a way that we  can remedy this situation. But it is worth pointing out that this completion approach also provides a foundation for developing finite dimensional approximations for the infinite dimensional model \eqref{M}.}

\subsection{ The $*$ Functor }
If $(X, \mathfrak{A}_X)$, $(Y,\mathfrak{A}_Y)$ are measurable spaces
let $\mathcal{M}(X)$,  $\mathcal{M}(Y)$ denote the finite signed
measures on $X$, $Y$. Let $*: X \mapsto X^* = \mathcal{M}_w(X)$
and if
$$\phi :X \rightarrow Y $$ is measurable let $$\phi^* : X^*
 \rightarrow Y^*$$ be given by $\phi^*(\mu)(E_Y) =\mu
( \phi^{-1}(E_Y))$. This * is a functor that will map dynamical
systems on $\mathcal{M}_w(X)$ to dynamical systems on
$\mathcal{M}_w(Y)$. Moreover due to the change of variable (see
Theorem \ref{changeofvar} in the Appendix) we see that if $\phi$ is
continuous where $\mathfrak{A}_X , \mathfrak{A}_Y$ denote the Borel
sets on the topological spaces $X$ and $Y$, then $$\phi^*: X^*
\rightarrow Y^* $$ is continuous. We collect this fact as a
proposition.

\begin{proposition} \label{wpreservation} If $ \mu(t) \to \mu_{0} $
 in the $ weak^*$ topology on $ X^*$, then $ \phi^*[\mu(t)] \to
\phi^*[\mu_{0}]$ in the $weak^*$ topology on $Y^*$ where
$\phi^*[\mu(t)](E_Y) = \mu(t)( \phi^{-1}(E_Y))$.
\end{proposition}

\subsection{Inverse Limits}
$I$ is called a directed set provided $\forall i,j \in I, \exists k
\in I$, such that $i,j \leq k$. If $(A_i)$ are sets, $(A_i, f_{ij})$
is called a directed system provided for $( i<j<k)$,
$$f_{ij}: A_i \to A_j$$ and $f_{jk}f_{ij} = f_{ik} $ .

Given a directed system one can form the inverse limit. Consider the
subset of $\prod A_i$ given by $$\lim_{\longleftarrow} A_i = \{
(x_i) \in \prod A_i: f_{ij}(x_i) =x_j\}.$$ For example, $( Z_{p^k},
\phi_{kl})$ is such a collection where  $ k
> l $ , $$\phi_{kl}: [x]_{p^k}\mapsto [x]_{p^l} $$ We have

$$Z_p \leftarrow Z_{p^2} \cdots Z_{p^k} \leftarrow Z_{p^{k+1}} \cdots
$$ The inverse limit of $( Z_{p^k}, \phi_{kl})$ is called the p-adic
numbers.
\begin{remark} We assume that all sets are directed and inverse
directed simply means the opposite direction, $a\leq_{op} b $, e.g.,
$a\leq_{op} b $ if and only if $a \geq b$.
\end{remark}
\subsection {Partitions of $Q$ and Completions of Measure Spaces }Let $(Q, \rho)$ be a compact Polish space, and let $
(\prod(Q), \preccurlyeq) $ denote the poset (partially ordered set)
of all partitions on $Q$. If $ \wp_1 , \wp_2 $ are two partitions of
$Q$, $ \wp_1 \preccurlyeq \wp_2 $ if each class in $ \wp_2 $ is a
union of classes in $ \wp_1 $. We denote the partition of singletons
by $\bot$, since it is the ``finest". The single partition where
everyone is related we denote by $\top$ since it is the ``coarsest".
Also, since every function on $Q$ determines a partition, we denote
this partition by $f$ if $f:Q \rightarrow \mathbb{Y}$ i.e., $f$ is
the partition $ \{f^{-1}(\{y\}): y \in \mathbb{Y} \}.$ If $\wp$ is a
partition on $Q$ we denote by $[q]_{\wp} $ the equivalence class of
$q \in Q$ and $ [Q]_{ \wp} =\{ [q]_{\wp}: q \in Q \} $. If
$[Q]_{\wp_1}$ is a topological space and $ \wp_1 \preccurlyeq \wp_2
$, then there is exactly one topology on $[Q]_{\wp_2}$ relative to
which the map $[q]_{\wp_1} \mapsto [q]_{\wp_2}$, denoted by
$\phi_{\wp_1, \wp_2}$, is a quotient map. We imbue $[Q]_{\wp_2}$
with this topology, i.e., $V$ is open in $[Q]_{\wp_2}$ if and only
if $\phi_{\wp_1, \wp_2}^{-1}(V)$ is open in $[Q]_{\wp_1}$. Since
$\bot$ is finer than all partitions, if $[Q]_{\bot}$ is a
topological space, then $[Q]_{\wp}$ is well defined as a topological
space for every partition ${\wp}$. We think of $[Q]_{\bot}$ as $Q$.
We do not distinguish between the two.

 { Let $I$ be a directed set. If
$ \{\wp_{i}\} _{ i \in I}$ is such that $ i \leq j $ if and only if
$\wp_i \preceq \wp_j $, then $\{\wp_{i}\} _{ i \in I}$ is called a
\textbf{distinguished} family of partitions. Any distinguished
family will generate a directed system of abelian groups of measure
spaces. Indeed, consider the following set of commuting diagrams. We
at times suppress notation by defining $\wp_i :=[Q]_{\wp_i}$, $
\phi_{i,j}:= \phi_{\wp_i, \wp_j}$ and $ \phi_{i}:= \phi_{\bot, i}$
(likewise for $\wp_i^*, \phi_{i,j}^*, \phi_{i}^*$), and
$\phi_{\bot,i}^*(\mu) =\mu^*_i.$}

 {
If $i < j$, then we have
\begin{equation}\label{diagram1} \xymatrix{
Q \ar[d]_{\phi_{\bot,i}} \ar[rd]^{\phi_{ \bot,j}} & \\
{\wp_i} \ar[r]_{\phi_{i,j}} & {\wp_j} }\end{equation} and applying
the $*$ functor we get
\begin{equation}\label{diagram2} \xymatrix{
Q^* \ar[d]_{\phi^*_{\bot,i}} \ar[rd]^{\phi^*_{\bot, j}} & \\
\phi_{\bot,i}^*(Q^*)_w \ar[r]_{\phi^*_{i,j}} &
\phi_{\bot,j}^*(Q^*)_w}. \end{equation}
 Since our family is distinguished, $\phi^*_{i,j}$ $: \mu^*_i \mapsto \mu^*_j$ is well defined.
  Also if $ i < j < k $, then
  $\phi^*_{j,k}\phi^*_{i,j} = \phi^*_{i,k}$ and we can form the inverse limit
  of $ (\{\phi_i^*(Q^*)_w\} _{ i \in I},
 \phi^*_{i,j}).$  Let  $ \{\wp_i \} _{i \in I} $ be distinguished and let $\mathbf{\wp}^*
=\prod_{ i \in I} \phi^*_i(Q^*)_{w}$ be given the natural
product vector space structure and the Tychonov topology
(coordinatewise convergence). Then we make the following definition.
\begin{definition}\label{WEAK} We define
 $\overline{\mathcal{M}}_{\wp^*}$ := $ \lim_{\longleftarrow}{ \phi^*_i(Q^*)_w}\cong
  \lim_{\longleftarrow} Q^*/ ker(\phi_i^*)_w
 $ since $Q^*/ ker(
 \phi_i^*) \cong \phi^*_i(Q^*) $ as abelian groups. It is the inverse limit of
 $ (\{\phi_i^*(Q^*)_w\} _{ i \in I},
 \phi^*_{i,j}).$ We call it the \textbf{completion} of $\mathcal{M}$ with respect to the family
$\{ \wp_i \}$, or the \textbf{completion} of $Q^*$ mod $\wp^*$.
\end{definition}}

The name ``completion" is apt in the
following sense. If $G$ is a group and $H_{k} \supseteq H_{k+1}$ is
a sequence of normal subgroups with finite index, then similar to a manner in which
the reals are constructed from the rationals we can complete a group
with respect to a family of subgroups. Indeed we define a sequence
$(g_n)$ in $G$ to be a Cauchy sequence provided for every $R$, there
exists a $N(R)$ such that for all $ n,m \geq N(R)$, $ x_nx_m^{-1}
\in H_R$. If for every $R$, there
exists a $N(R)$ such that for all $ n\geq N(R)$, $ x_n
\in H_R$ we call this a \textbf{null sequence}. The Cauchy sequences form a group under termwise multiplication and the null sequences form a normal subgroup. The factor group, $\widehat{G}$,
is called the completion of $G$ (with respect to the family
$H_k$). From \cite[pg. 52]{Serge2} we have that $$\widehat G \cong
\lim_{\longleftarrow} G/H_k
$$ and $$ G \rightarrow  \lim_{\longleftarrow} G/H_k \quad \text{  by } \quad g \mapsto (g,g,
...).$$

For our purposes we notice that $Q^*$ is an Abelian group. If $I$ is
countable, and a total order such that $\wp_{i+1}\preceq \wp_i$,
then $ Q^* \supseteq ker(\phi_{i}^*) \supseteq ker( \phi^*_{i+1}).$ According to \cite[pg. 52]{Serge2} if each $ker(\phi_{i}^*)$ is of finite index, we can form the completion of $Q^*$ with respect to the
sequence of groups $\{ ker(\phi_i^*) \}$, which is isomorphic to $
\lim_{\longleftarrow} Q^*/ker(\phi_i^*).$ We also have \cite[pg.
52]{Serge2}
$$\mathcal{M}=Q^* \rightarrow \overline{\mathcal{M}}_{\wp^*},$$
given by $\mu \mapsto ( \mu, \mu,\mu,...)$. We summarize a few facts below regarding  $\overline{\mathcal{M}}_{\wp^*}$ in particular we show that it can be used as an extended space which contains the measures.

\subsection{A Fundamental Theorem}
\begin{theorem}\label{FT}
Let $\mathbf{\wp}^* $ be as in the paragraph before definition \ref{WEAK} above. Let
$$ \mathbf{\Phi}^*=\prod \phi_{\bot,i}^*: Q_w^* \rightarrow
\overline{\mathcal{M}}_{\wp^*}$$ be such that $ \mu \mapsto \prod
\phi_{\bot,i}^*( \mu)$. Then for $\mathbf{\Phi}^*$ we have:
\begin{enumerate}
\item  $\mathbf{\Phi}^*$ is a $weak^*$ continuous algebraic (vector space)
homomorphism.
\item  $\mathbf{\Phi}^*$ preserves the cone.
\item $Q^*/ ker\{ \Phi^*\} $ is isomorphic to a subspace
of $\overline{\mathcal{M}}_{\wp^*}$.
\item  There exists a countable
sequence $\{P_i\}_{i=1}^\infty$ of finite partitions such that $
P_{i+1} \preceq P_i $, each element of $P_i$ is a $G_{\delta} $ (countable intersection of open sets) with
diameter smaller than $\frac{1}{i} $, and $\Phi^*$ is an embedding
(algebraic).
\end{enumerate}

\end{theorem}

\begin{proof} The proofs of
1-3 are trivial and will be left to the reader. As for 4, the
existence of such a $\{P_i\}_{i=1}^\infty$ is a standard
construction and can be found in \cite[pg.66]{Arveson}. We need only
show that such a partition generates an embedding on the cone. So
let $(P_i,\{a_{i,n}\}_{n=1}^{n=N_i})_{i\in Z}$ be such that each
$P_i$ is such a partition, $N_i$ is the number of elements in the
$ith$ partition and $a_{i,n}$ is an element of the $nth$ partition
element of $P_i$.

We notice that $\mu_i^* =
\sum_{k=1}^{N_i}\mu(P_{i,k})\delta_{P_{i,k}}$ and from here we can
form an $ith$ approximation measure for $\mu$:

$$\mu^{(i)} = \sum_{k=1}^{N_i}\mu(P_{i,k})\delta_ {a_{i,k}} .$$
 Then
$\mu^{(i)} \rightarrow \mu$ in the $weak^*$ topology. Hence if $\mu$
and $\nu$ are positive and such that $\mu_i^*=\nu_i^*$ for all $i$,
then $ \mu =\nu $ and $\Phi^*$ is an embedding.

 Let $f \in C(Q)$. We must show that $
\rho_f(\mu^{(i)} -\mu) \rightarrow 0$, as $i \rightarrow \infty.$
Since $Q$ is compact, $f$ is uniformly continuous. So if $\epsilon
>0$, there is a $\delta
>0$ such that $| f(q) -f(\hat q)| < \epsilon$, if $\rho(q, \hat q) <
\delta.$ Next pick $i$ such that $\frac{1}{i} < \delta$. Then we
have the following estimates:
$$ \int_{P_{i,k}}(f(a_{i,k})-\epsilon) d\mu \leq \int_{P_{i,k}}f(q)d\mu
\leq \int_{P_{i,k}}(f(a_{i,k})+ \epsilon) d\mu .$$ Hence

$$ \mu(P_{i,k})(f(a_{i,k})-\epsilon)  \leq \int_{P_{i,k}}f(q)d\mu
\leq \mu(P_{i,k})(f(a_{i,k})+ \epsilon) $$ and

$$ \sum _{k=1}^{N_i}\mu(P_{i,k})(f(a_{i,k})-\epsilon)  \leq \sum _{k=1}^{N_i} \int_{P_{i,k}}f(q)d\mu
\leq \sum _{k=1}^{N_i}\mu(P_{i,k})(f(a_{i,k})+ \epsilon). $$ Since
the $P_{i,k}$ 's form a partition,

$$  -\epsilon \mu(Q) \leq \sum _{k=1}^{N_i}[-\mu(P_{i,k})f(a_{i,k})+
\int_{P_{i,k}}f(q)d\mu] \leq \mu(Q) \epsilon. $$ Hence $\rho_f(\mu^{(i)} -\mu) < \epsilon \mu(Q)$ for arbitrary $ \epsilon
>0.$
We get our full result by decomposing each finite signed measure
into positive and negative parts and noticing that $\phi_i^*$ is
linear.
\end{proof}

Theorem \ref{FT} prompts the following definition.
\begin{definition}\label{NA} Let $\wp^*$ be as above. Let $(\mu_i^*)_{i\in I}$ be the
image of $\mu$ under $\Phi^*$. Then each $\mu_i^*$ is called the
\textbf{ith mod $\wp^*$ approximation} to $\mu$.
\end{definition}
Definition  \ref{NA} is particularly interesting in light of part 4
in Theorem \ref{FT}. This theorem implies that with this setup one
can reduce the infinite dimensional problems of asymptotic analysis
and numerical approximation to countably many finite problems.  For
each $i$, studying $\phi_{\bot,i}^*(\mu)[t]= \mu^*_i[t]$, reduces to
studying a finite system of ordinary differential equtions since $\wp_i$ is finite. If we
choose each $a_{i,k}$ to be a member of some countable dense subset
like the diadic rationals, then we can use
\begin{equation} \label{approximation} \mu^{(i)}(t) =
\sum_{k=1}^{N_i}\mu(t)(P_{i,k})\delta_ {a_{i,k}} \end{equation} as
an approximation which we can numerically compute. Equation
\eqref{approximation} will generate a finite system of ordinary differential equations. So for
$i$ large enough, if we use standard tools to understand this ordinary differential equation
system, then we have a good approximation to the asymptotic behavior
of $\mu(t)$ and each such subsystem is an approximation to $\mu(t)$.
We will illustrate this technique for numerical analysis in future
work.
 \subsection{ Weak Asymptotic Limits}
The fundamental Theorem \ref{FT} prompts the following definition.

\begin{definition} Let $\wp^*$ be as above. If there exists a family $(\nu_i)_{i\in I} \in \mathbf{\wp}^*$
such that
$$ \lim_{t \rightarrow \infty} \phi^*_{\bot,i} \mu(t) = \nu_i $$
for all $i$, then we call $(\nu_i)_{i\in I}$ the \textbf{asymptotic
limit mod} ${\wp^*} $ of $ \mu(t) $.
\end{definition}
From Proposition \ref{wpreservation} we see that if $\mu(t)$ has an
asymptotic limit in $Q^*$, then it has one in $\wp_i^*$ for each
$i$. This shows that the \textbf{asymptotic limit mod} $\wp^*$ is
indeed a generalized or weak solution, since it agrees with the
classical one if a classical one exists.
 {
Even if $\Phi^*$ is not an embedding, one cannot understand the
large dimensional ordinary differential equation system, or if $\mu(t)$ does not have a limit,
one can still utilize Theorem \ref{FT} to find very meaningful
information about the asymptotic and numerical behavior of the
model. In other words even if $\Phi^*$ is not an embedding, for
every $\wp^*$, the mod $\wp^*$ limit tells us something. These
solutions are called weak or generalized solutions and there are
some important and canonical solutions of these type.}

 {
So the weak problem becomes. Can we find $\{\wp_i\}$ such that we
can solve the weak problem? Solve
$$\lim_ {t \rightarrow \infty} \Phi^* \mu = \mathfrak{\nu}, \quad
\nu \in \overline{\mathcal{M}}_{\wp^*}.$$ as a global equilibrium.}
 {
 Given an EGT model, often a natural partition will emerge and there are some that are canonical.
In particular, for pure replicator dynamics there were three canonical limits.
 They are the mod $\bot$, mod $\top$, and the mod $
\mathcal{R}(0, \cdot)$ limits. The mod $\bot$ limit is clearly the
asymptotic limit of the model, if it exists. The mod $\top$ limit
will always give the equilibrium of the total population. The mod
$\mathcal{R}(0, \cdot)$ is important specifically for pure replicator dynamics as it will always converge to a
Dirac mass centered at the fittest class.}

\section{ \bf Concluding Remarks} We studied the long time behavior of an EGT (selection-mutation)
 model on the space of measures. We showed that in the pure replication case the asymptotic behavior manifest itself
 into a Dirac measure centered at the fittest class (even if it is not unique). We also showed that if the initial
 condition and the selection-mutation kernel are discrete then there is a globally asymptotically stable equilibrium that attracts
 all solutions; provided mutation is small.

Epidemic models which consider the dynamics of multi-strain pathogens have been studied in the literature (e.g., \cite{AA1,BT}). These models
have been formulated as systems of ordinary differential equations where infected individual are distributed over a set of $n$ classes each carrying a particular
strain of finite and discrete strain (strategy) space. However, often a continuous (strategy) space is needed. For example, think
of particular disease that has transmission rate $\beta$ with possible values in the interval $[\underline \beta,\overline \beta]$. Then infected individuals
are distributed over a strain space with transmission taking values in this interval. The formulation and the general theory developed here have potential application in treating such distributed rate epidemic models.

While the above are the main contributions, it is worth mentioning that
the present work also provides a foundation for developing  numerical
methods to compute solutions of these models and to estimate
parameters from field data. For numerical analysis and parameter estimation the completion of the space of measures
presented here
reduces the investigation of the infinite dimensional problem to countably many
finite dimensional problems (see Theorem 3.5). These ideas will be explored in more
details elsewhere.\\

 \vspace{0.0 in}
 \noindent {\bf Acknowledgements:} The authors would like to thank Horst
 Thieme for thorough reading of an earlier version of this
 manuscript and for the many useful comments, in particular, for suggesting Theorem
 \ref{persistence}. The authors would also like to thank their colleague Ping Ng
 for helpful discussions. This work was partially supported by the National Science Foundation under grant \# DMS-0718465.

\section{ \bf Appendix}

 For the convenience of the reader, we next state a few known results that are used in our analysis.

 The next theorem is concerned with
\begin{equation}\label{GLEQ}
 x'=f(t,x),~~ x(t_0) = x_0, \end{equation} where $f \in C[\mathbb{R}_+ \times E,
E]$, $E$ being a Banach space.

Consider
\begin{equation}\label{finite}
\dot{x} = f(x, \lambda),
\end{equation}
where $ f: U \times \Lambda \rightarrow \mathbb{R}^n $ is
continuous, $ U \subseteq \mathbb{R}^{n}$, $\Lambda \subseteq
\mathbb{R}^{k}$ and $ D_{x}f (x,\lambda) $ is continuous on $ U
\times \Lambda $. We write $x(t, z, \lambda) $ for the solution of
\eqref{finite} satisfying $x(0)=z.$
\begin{theorem}\label{pert1} \cite{SmithWalt}
Assume that $(x_0, \lambda_{0}) \in U\times\Lambda $, $x_0 \in
int(U)$ , $f(x_0, \lambda_0) =0$, all eigenvalues of $D_{x}f(x_0,
\lambda_0)$  have negative real part, and $x_0$ is globally
attracting in $U$ for solutions of \eqref{finite} with $\lambda
=\lambda_0 $. If
\begin{description}
\item (H1) there exists a compact set $ D \subseteq U$  such that
for each $ \lambda \in \Lambda$ and each $ z \in U$, $x(t,z,
\lambda) \in D$ for all large t,
\end{description}
then there exists $\epsilon >0$ and a unique point $\widehat{x}(
\lambda) \in U$ for $\lambda \in B_{\lambda}( \lambda_0, \epsilon)$
such that $f(\widehat{x}( \lambda),\lambda)=0$ and $x(t,z, \lambda)
\rightarrow \widehat{x}(\lambda)$, as $t\rightarrow \infty$ for all
$z \in U$.

\begin{theorem} \cite[pg. 393]{Royden}\label{changeofvar} Let $\varphi$ be a measurable point mapping of the
measure space $<X,\mathfrak{A},\mu>$ into the measurable space
$<Y,\mathfrak{B}> $. Let $\Phi^*$ be the induced set mapping of
$\mathfrak{B}$ into $\mathfrak{A}$. Then for each nonnegative
measurable function $f$ on $Y$ we have $$ \int_Y f d\Phi^{*}\mu =
\int_{X} f\circ\varphi d\mu .$$ Here $ \Phi^*\mu( E_Y) = \mu
(\Phi(E_Y)) $ and $ \Phi(E_Y) = \varphi^{-1}(E_Y)). $
\end{theorem}

\begin{theorem} \label{AdjointTheorem} \cite[pg. 140]{Munkres}
Let $g: X \rightarrow Z$ be a surjective continuous map. Let $ X^*$
be the following collection of subsets of $X$,  $$ X^* = \{
g^{-1}(\{z\}): z \in Z \} .$$

Give $X^*$ the quotient topology.

\begin{itemize}
\item[(a)]If $Z$ is Hausdorff, so is $X^*$.
\item [(b)] The map $g$ induces a bijective continuous map $f: X^*
\rightarrow Z $, which is a homeomorphism if and only if $g$ is a
quotient map.
$$\xymatrix{
X \ar[d]_{p} \ar[rd]^g & \\
X^* \ar[r]_f & Z }$$

\end{itemize} \end{theorem}
\end{theorem}

\newpage

  \begin{figure}[htbp] \label{Fig1}
 \resizebox{12cm}{!}{\includegraphics{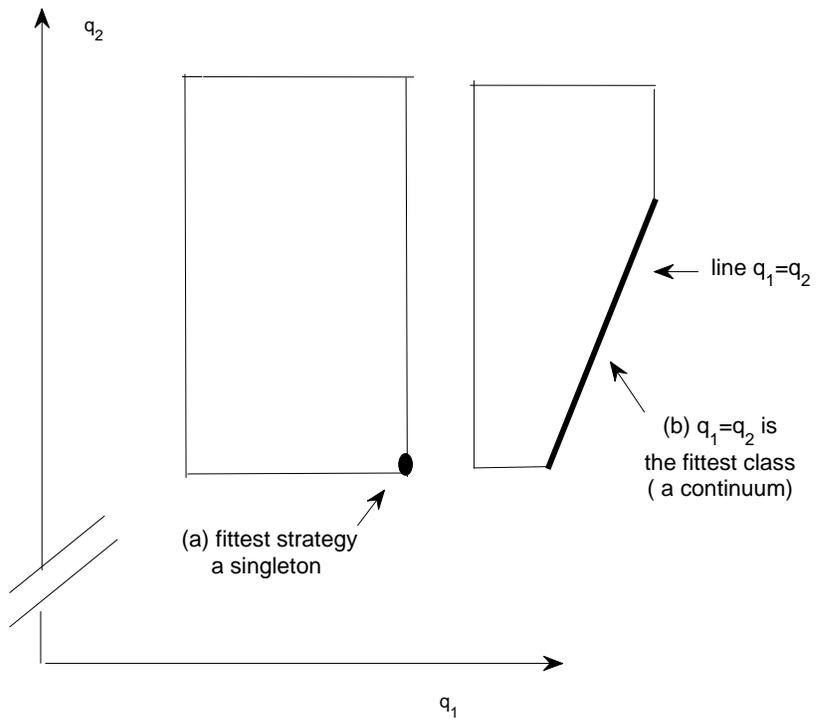}}
 \caption{ Two examples of strategy spaces.}
\label{strategyspace}
\end{figure}

\begin{figure}[htbp] \label{Fig2}
 \resizebox{12cm}{!}{\includegraphics{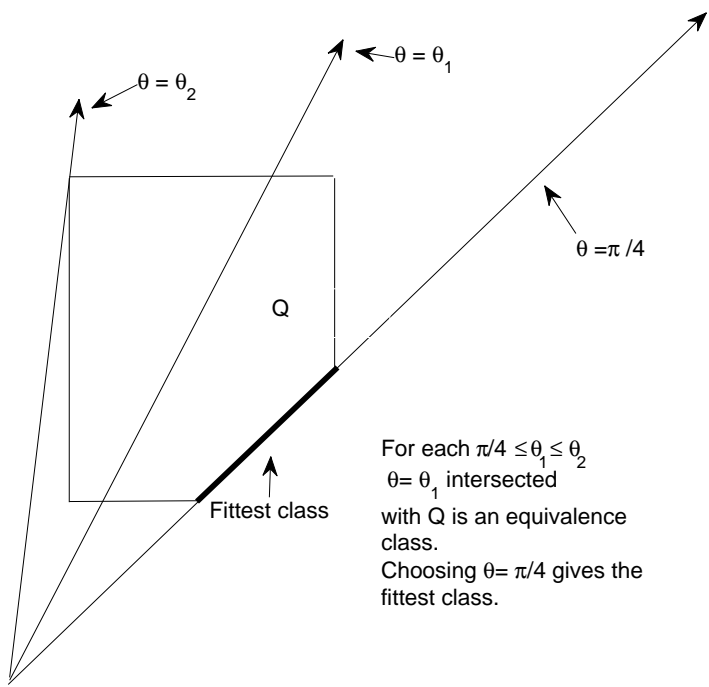}}
 \caption{Example of the $[Q]_{\mathcal{R}}$ partition.}
\label{partition}
\end{figure}
\end{document}